\documentclass[11pt,oneside,english]{amsart}
\usepackage[a4paper]{geometry}
\usepackage{color}
\usepackage{amsmath}
\usepackage{amstext}
\usepackage{amsthm}
\usepackage{amssymb}
\usepackage{setspace}
\usepackage{hyperref}
\usepackage{esint}
\usepackage[foot]{amsaddr}
\makeatletter
\numberwithin{equation}{section}
\numberwithin{figure}{section}
\theoremstyle{plain}
\newtheorem{thm}{\protect\theoremname}[section]
  \theoremstyle{definition}
  \newtheorem{defn}[thm]{\protect\definitionname}
  \theoremstyle{plain}
  \newtheorem{prop}[thm]{\protect\propositionname}
  \theoremstyle{plain}
  \newtheorem{cor}[thm]{\protect\corollaryname}
  \theoremstyle{remark}
  \newtheorem*{acknowledgement*}{\protect\acknowledgementname}
  \theoremstyle{plain}
  \newtheorem{lem}[thm]{\protect\lemmaname}
  \theoremstyle{plain}
  \newtheorem*{thm*}{\protect\theoremname}
  \theoremstyle{remark}
  \newtheorem{rem}[thm]{\protect\remarkname}
  \theoremstyle{plain}
  \newtheorem{question}[thm]{\protect\questionname}
\newtheorem{thmx}{Theorem}

\newtheorem{example}[thm]{Example}

\makeatother

  \providecommand{\acknowledgementname}{Acknowledgement}
  \providecommand{\corollaryname}{Corollary}
  \providecommand{\definitionname}{Definition}
  \providecommand{\lemmaname}{Lemma}
  \providecommand{\propositionname}{Proposition}
  \providecommand{\questionname}{Question}
  \providecommand{\remarkname}{Remark}
  \providecommand{\theoremname}{Theorem}

\def\Qp{\mathbb Q_p}
\def\Zp{\mathbb{Z}_p}

\def\val{\mathrm{val}}
\def\ac{\mathrm{ac}}

\def\complex{\mathbb{C}}
\def\reals{\mathbb{R}}
\def\nats{\mathbb{N}}
\def\ints{\mathbb{Z}}
\def\rats{\mathbb{Q}}
\def\supp{\mathrm{supp}}

\def\VF{\mathrm{VF}}
\def\RF{\mathrm{RF}}
\def\VG{\mathrm{VG}}

\def\spec{\mathrm{Spec}}
\def\sinc{\mathrm{sinc}}
\def\Loc{\mathrm{Loc}}
\def\Ldp{\mathcal{L}_\mathrm{DP}}

\AtEndDocument{\bigskip{\footnotesize%
  \textsc{Itay Glazer, Department of Mathematics, Northwestern University,  2033 Sheridan Rd., Evanston IL 60208, USA.
} \\
  \textit{E-mail address}: \texttt{itayglazer@gmail.com} \par
  \addvspace{\medskipamount}
}}
\AtEndDocument{\bigskip{\footnotesize%
  \textsc{Yotam Hendel, Department of Mathematics, Northwestern University,  2033 Sheridan Rd., Evanston IL 60208, USA.
} \\
  \textit{E-mail address}: \texttt{yotam.hendel@gmail.com} \par
  \addvspace{\medskipamount}
}}
\AtEndDocument{\bigskip{\footnotesize%
  \textsc{Gady Kozma, Faculty of Mathematics and Computer Science, the  Weizmann Institute of Science, 234 Herzl St., Rehovot 76100, ISRAEL.
} \\
  \textit{E-mail address}: \texttt{gady.kozma@weizmann.ac.il} \par
  \addvspace{\medskipamount}
}}

\title[On singularity properties of convolutions of algebraic morphisms]{On singularity properties of convolutions of algebraic morphisms
- the general case}
\author{Itay Glazer and Yotam I. Hendel (with an appendix joint with Gady
Kozma)}
\subjclass[2010]{14B05, 14E18, 03C98 (primary), 11G25, 14G05, 20G15 (secondary).}
\begin{document}
\maketitle
\begin{abstract}
Let $K$ be a field of characteristic zero, $X$ and $Y$ be smooth
$K$-varieties, and let $G$ be an algebraic $K$-group. Given two
algebraic morphisms $\varphi:X\rightarrow G$ and $\psi:Y\rightarrow G$,
we define their convolution $\varphi*\psi:X\times Y\to G$ by $\varphi*\psi(x,y)=\varphi(x)\cdot\psi(y)$.
We then show that this operation yields morphisms with improved smoothness
properties. More precisely, we show that for any morphism $\varphi:X\rightarrow G$
which is dominant when restricted to each geometrically irreducible
component of $X$, by convolving it with itself finitely many times,
one obtains a flat morphism with reduced fibers of rational singularities.
Uniform
bounds on families of morphisms are given as well. Moreover, as a
key analytic step, we also prove the following result in motivic integration;
if $\{f_{\mathbb{Q}_{p}}:\mathbb{Q}_{p}^{n}\rightarrow\mathbb{C}\}_{p\in\mathrm{primes}}$
is a collection of motivic functions, and $f_{\mathbb{Q}_{p}}$ is
$L^{1}$ for any $p$ large enough, then in fact there exists $\epsilon>0$
such that $f_{\mathbb{Q}_{p}}$ is $L^{1+\epsilon}$ for any $p$
large enough. 
\end{abstract}

\pagenumbering{arabic} \tableofcontents{}

\global\long\def\nats{\mathbb{N}}%
 
\global\long\def\reals{\mathbb{\mathbb{R}}}%
 
\global\long\def\ints{\mathbb{Z}}%
 
\global\long\def\val{\mathbb{\mathrm{val}}}%
 
\global\long\def\Qp{\mathbb{Q}_{p}}%
 
\global\long\def\Zp{\mathbb{\mathbb{Z}}_{p}}%
 
\global\long\def\ac{\mathbb{\mathrm{ac}}}%
 
\global\long\def\complex{\mathbb{\mathbb{C}}}%
 
\global\long\def\rats{\mathbb{\mathbb{Q}}}%
 
\global\long\def\supp{\mathbb{\mathrm{supp}}}%
 
\global\long\def\VF{\mathbb{\mathrm{VF}}}%
 
\global\long\def\RF{\mathbb{\mathrm{RF}}}%
 
\global\long\def\VG{\mathbb{\mathrm{VG}}}%
 
\global\long\def\spec{\mathbb{\mathrm{Spec}}}%
 
\global\long\def\Ldp{\mathbb{\mathcal{L}_{\mathrm{DP}}}}%

\raggedbottom

\section{Introduction}

In this paper we continue the investigation of the algebraic convolution
operation, which was initiated in \cite{GH19}. Namely, we study the
following operation: 
\begin{defn}
\label{def:convolution} Let $X_{1}$ and $X_{2}$ be algebraic varieties,
$G$ an algebraic group and let $\varphi_{1}:X_{1}\to G$ and $\varphi_{2}:X_{2}\to G$
be algebraic morphisms. We define their convolution by 
\begin{gather*}
\varphi_{1}*\varphi_{2}:X_{1}\times X_{2}\to G\\
\varphi_{1}*\varphi_{2}(x_{1},x_{2})=\varphi_{1}(x_{1})\cdot\varphi_{2}(x_{2}).
\end{gather*}
In particular, the $n$-th convolution power of a morphism $\varphi:X\rightarrow G$
is 
\[
\varphi^{*n}(x_{1},\ldots,x_{n}):=\varphi(x_{1})\cdot\ldots\cdot\varphi(x_{n}).
\]
\end{defn}

The convolution operation as above can be viewed as a geometric version
of the classical convolution as follows. Firstly, recall that given
$f_{1},f_{2}\in L^{1}(\reals^{n})$, their convolution is defined
by 
\[
(f_{1}\ast f_{2})(x)=\int_{\reals^{n}}f_{1}(t)f_{2}(x-t)dt,
\]
and it has improved smoothness properties, e.g. 
\begin{itemize}
\item if $f_{1}\in C^{k}(\reals^{n})$ and $f_{2}\in C^{l}(\reals^{n})$,
then $(f_{1}*f_{2})'=f_{1}'*f_{2}=f_{1}*f_{2}'$ and therefore $f_{1}*f_{2}\in C^{k+l}(\reals^{n})$. 
\item In particular, if $f_{1}$ is smooth, then $f_{1}*f_{2}$ is smooth
for every $f_{2}\in L^{1}(\reals^{n})$. 
\end{itemize}
Now, for morphisms $\varphi_{i}:X_{i}\to G$ for $i=1,2$ as before,
consider the functions 
\[
F_{\varphi_{i}}:G\to\mathrm{Schemes}\text{ by }F_{\varphi_{i}}(g)=\varphi_{i}^{-1}(g).
\]
Given a finite ring $A$, we naturally get maps $(\varphi_{i})_{A}:X_{i}(A)\to G(A)$
and $(\varphi_{1}*\varphi_{2})_{A}:X_{1}(A)\times X_{2}(A)\to G(A)$
from finite sets to the finite group $G(A)$, and furthermore, 
\[
(\varphi_{1}*\varphi_{2})_{A}^{-1}(s)=\biguplus\limits _{g\in G(A)}(\varphi_{1})_{A}^{-1}(g)\times(\varphi_{2})_{A}^{-1}(g^{-1}s).
\]
In particular, if we set $|F_{(\varphi_{i})_{A}}|:G(A)\to\ints_{\geq0}$
to be the function counting the size of the fibers, i.e.~$|F_{(\varphi_{i})_{A}}|(g)=|(\varphi_{i})_{A}^{-1}(g)|$,
we see that the algebraic convolution operation commutes with counting
points over finite rings: 
\[
|F_{(\varphi_{1})_{A}}|*|F_{(\varphi_{2})_{A}}|(s)=\sum_{g\in G(A)}|F_{(\varphi_{1})_{A}}|(g)\cdot|F_{(\varphi_{2})_{A}}|(g^{-1}s)=|F_{(\varphi_{1}*\varphi_{2})_{A}}|(s).
\]
It is thus natural to ask whether analogously to the analytic convolution
operation, the algebraic convolution operation improves smoothness
properties of morphisms: \begin{question}\label{que:(convolution}Let
$\varphi_{i}:X_{i}\rightarrow G$ for $i=1,2$ be two morphisms from
varieties $X_{1}$ and $X_{2}$ to an algebraic group $G$, and assume
that $\varphi_{1}$ satisfies a singularity property $S$. 
\begin{enumerate}
\item When does $\varphi_{1}*\varphi_{2}$ have property $S$ as well? 
\item Which singularity properties can one obtain after finitely many self-convolutions
of $\varphi_{1}$? 
\end{enumerate}
\end{question} Concerning (1), the following proposition shows that
the convolution operation preserves singularity properties of morphisms
in the following sense: 
\begin{prop}[{\cite[Proposition 3.1]{GH19}}]
\label{prop:convpreservesgoodproperties} Let $X$ and $Y$ be varieties
over a field $K$, let $G$ be an algebraic group over $K$ and let
$S$ be a property of morphisms that is preserved under base change
and compositions. If $\varphi:X\to G$ is a morphism that satisfies
the property $S$, the natural map $i_{K}:Y\to\mathrm{Spec}(K)$ has
property $S$ and $\psi:Y\to G$ is arbitrary, then $\varphi*\psi$
and $\psi*\varphi$ has property $S$. 
\end{prop}

The rest of this paper is devoted to the study of the second part
of Question \ref{que:(convolution}, and generalizes the results from
\cite{GH19} in which the case where $G=V$ is a vector space was
dealt with.

If $\varphi$ is not smooth, then one in general can not guarantee
that some convolution power of $\varphi$ will be smooth (e.g. $\varphi:\mathbb{A}^{1}\to(\mathbb{A}^{1},+)$
via $x\mapsto x^{2}$, see Proposition \ref{prop:not smooth after convolutions}
for a more general statement). However, it is possible to achieve
other singularity properties as in Theorems \ref{Main result} and
\ref{thm: singularity properties obtained after convolution} in the
next section.

From here henceforth let $K$ denote a field of characteristic $0$.
The following property plays a key role in this paper: 
\begin{defn}[{The (FRS) property}]
\label{def:(FRS)} Let $X$ and $Y$ be smooth $K$-varieties. We
say that a morphism $\varphi:X\rightarrow Y$ is (FRS) if it is flat
and if every fiber of $\varphi$ is reduced and has rational singularities
(for rational singularities see Definition \ref{defn: rational sings}). 
\end{defn}

The (FRS) property was first introduced in \cite{AA16}, where it
was proved that for any semi-simple algebraic group $G$ the commutator
map $[\cdot,\cdot]:G\times G\to G$ is (FRS) after $21$ self-convolutions.
This was then used to show in \cite{AA16} and \cite{AA18} that if
$\Gamma$ is a compact $p$-adic group or an arithmetic group of higher
rank then its representation growth is polynomial and does not depend
on $\Gamma$. Explicitly, for $\Gamma$ as above and every $c>40$
it holds that 
\[
r_{n}(\Gamma):=\#\{\text{irreducible }n\text{-dimensional }\complex\text{-representations of }\Gamma\text{ up to equivalence}\}=o(n^{c}).
\]
It was furthermore proved in \cite{AA18}, based on works of Denef
\cite{Den87} and Musta\c{t}\u{a} \cite{Mus01}, that each individual fiber of an (FRS)
morphism has good asymptotic point count over finite rings of the
form $\ints/p^{k}\ints$ (either in $p$ or in $k$, see \cite[Theorem A]{AA18}
and \cite[Theorem 1.4]{Gla19}). 
In an upcoming work, we prove a version of this result which is uniform as the fiber varies. 
This allows one to interpret Question
\ref{que:(convolution}(2) with respect to the (FRS) property in a
probabilistic way: given $\varphi:X\to G$, then the (FRS) property
of $\varphi^{*n}$ can be reformulated in terms of uniform $L^{\infty}$-boundedness
after $n$ steps, of a family of random walks on $\{G(\ints/p^{k}\ints)\}_{p,k}$,
which is obtained by pushing forward the family of uniform probability
measures on $\{X(\ints/p^{k}\ints)\}_{p,k}$ under $\varphi$. 
In particular, the results and methods 
of this paper are used in \cite{GHb} to investigate families of random walks on compact $p$-adic groups which are induced by words\footnote{By a word $w$ we mean an element of the free group  on $r$  generators. Such $w$ induces a map $w_G:G^r\to G$ on any group $G$. If $G$ is an algebraic group, $w_G$ is an algebraic morphism and we can use algebraic geometry to study it.}.
For a precise discussion of this probabilistic interpretation see  \cite[Section 9]{GHb}, and for further discussion of the (FRS) property and its implications
see \cite[Section 1.3]{GH19} or  \cite{AA18,AA16}.

\subsection{Main results}

\subsubsection{Algebro-geometric results}
\begin{defn}
We say that a $K$-morphism $\varphi:X\to Y$ is \textit{strongly
dominant} if it is dominant when restricted to each geometrically
irreducible component of $X$. 
\end{defn}

In this paper we verify a conjecture of Aizenbud and Avni (see \cite[Conjecture 1.6]{GH19}),
showing that every strongly dominant morphism into an algebraic group
becomes (FRS) after finitely many self-convolutions: \begin{thmx}[Theorem
\ref{Main result v2}] \label{Main result} Let $X$ be a smooth
$K$-variety, $G$ be a connected algebraic $K$-group and let $\varphi:X\to G$
be a strongly dominant morphism. Then there exists $N\in\mathbb{N}$
such that for any $n>N$, the $n$-th convolution power $\varphi^{*n}$
is (FRS). \end{thmx}

Note that by \cite[Theorem 3.4]{AA16} and \cite[Theorem 1.11]{Rei18},
Theorem \ref{Main result} implies the following: 
\begin{cor}
Let $F$ be a local field of characteristic $0$. Let $X$ be a smooth
$F$-variety, $G$ be a connected algebraic $F$-group and let $\varphi:X\to G$
be a strongly dominant morphism. Then there exists $N\in\nats$ such
that for every $n>N$ and any smooth\footnote{{See Definition \ref{def: smooth and locally Ls} for the case where $F$ 
 is a non-Archimedean local field.}}, compactly supported
measure $\mu$ on $X(F)$, the $n$-th convolution power $\tau^{*n}$ of the pushforward $\tau:=\varphi_{*}(\mu)$
 has continuous density.
\end{cor}

It is easy to see that one cannot give a universal bound (i.e.~independent
of the map) on the number of convolutions needed in order to obtain
an (FRS) morphism. For example, the morphism $\varphi(x)=x^{n}$ requires
$n+1$ self convolutions in order to become an (FRS) morphism. For
other properties as below, an upper bound depending only on $\mathrm{dim}G$
can be given:

\begin{thmx}[see Propositions \ref{prop: upper bounds for properties}
and \ref{prop:The-bounds-are tight}] \label{thm: singularity properties obtained after convolution}
Let $m\in\nats$, let $X_{1},\ldots,X_{m}$ be smooth $K$-varieties,
let $G$ be a connected algebraic $K$-group, and let $\{\varphi_{i}:X_{i}\rightarrow G\}_{i=1}^{m}$
be a collection of strongly dominant morphisms. 
\begin{enumerate}
\item For any $1\leq i,j\leq m$ the morphism $\varphi_{i}*\varphi_{j}$
is surjective. 
\item If $m\geq\mathrm{dim}G$ then $\varphi_{1}*\ldots*\varphi_{m}$ is
flat. 
\item If $m\geq\mathrm{dim}G+1$ then $\varphi_{1}*\ldots*\varphi_{m}$
is flat with reduced fibers. 
\item If $m\geq\mathrm{dim}G+2$ then $\varphi_{1}*\ldots*\varphi_{m}$
is flat with normal fibers. 
\item If $m\geq\mathrm{dim}G+k$, with $k>2$, then $\varphi_{1}*\ldots*\varphi_{m}$
is flat with normal fibers which are regular in codimension $k-1$. 
\end{enumerate}
Furthermore, these bounds are tight. \end{thmx}

It is a consequence of \cite{Elk78} and \cite[Corollary 2.2]{AA16},
that the (FRS) property is preserved under small deformations. This
allows us to extend our main result, Theorem \ref{Main result}, to
families of morphisms (thus generalizing \cite[Theorem 7.1]{GH19}):
\begin{thmx}\label{main result for families}Let $K$ and $G$ be
as in Theorem \ref{Main result}, let $Y$ be a $K$-variety, let
$\widetilde{X}$ be a family of varieties over $Y$, and let $\widetilde{\varphi}:\widetilde{X}\rightarrow G\times Y$
be a $Y$-morphism. Denote by $\widetilde{\varphi}_{y}:\widetilde{X}_{y}\rightarrow G$
the fiber of $\widetilde{\varphi}$ at $y\in Y$. Then, 
\begin{enumerate}
\item The set $Y':=\{y\in Y:\widetilde{X}_{y}\text{ is smooth and }\widetilde{\varphi}_{y}:\widetilde{X}_{y}\rightarrow G\text{ is strongly dominant}\}$
is constructible. 
\item There exists $N\in\nats$ such that for any $n>N$, and any $n$ points
$y_{1},\ldots,y_{n}\in Y'$, the morphism $\widetilde{\varphi}_{y_{1}}*\dots*\widetilde{\varphi}_{y_{n}}:\widetilde{X}_{y_{1}}\times\dots\times\widetilde{X}_{y_{n}}\rightarrow G$
is (FRS). 
\end{enumerate}
\end{thmx}

As a consequence, we deduce the following theorem: \begin{thmx}[{See
\cite[Definition 7.7]{GH19} for the definition of complexity}] \label{FRS on complexity}
Let $G$ be an algebraic $K$-group. For any $\mathrm{dim}G<D\in\nats$,
there exists $N(D)\in\nats$ such that for any $n>N(D)$ and $n$
strongly dominant morphisms $\{\varphi_{i}:X_{i}\rightarrow G\}_{i=1}^{n}$
of complexity at most $D$ where $\{X_{i}\}_{i=1}^{n}$ are smooth
$K$-varieties, the morphism $\varphi_{1}*\dots*\varphi_{n}$ is (FRS).
\end{thmx}

\subsubsection{Model-theoretic/analytic results}

The heart of the proof of Theorems \ref{Main result} and \ref{main result for families}
lies in proving the model theoretic statements Theorems \ref{Main model theoretic result for vector spaces},
\ref{Main model theoretic result-Varieties}, \ref{Main model theoretic result for families of vector spaces}
and \ref{Model theoretic result- families of varieties}, which are
interesting on their own merits. We briefly explain this connection.

Let $\Ldp$ denote the first order Denef-Pas language (see Section
\ref{subsec:The-Denef-Pas-language,}) and let $\mathrm{Loc}$ denote
the collection of all non-Archimedean local fields. We also use the
notation $F\in\mathrm{Loc}_{>}$ to denote ``$F\in\mathrm{Loc}$
with large enough residual characteristic''. Given an algebraic $\rats$-variety
$X$, its ring of $\Ldp$-motivic functions $\mathcal{C}(X)$ and
the notion of an $\Ldp$-motivic measure can be defined (e.g. see~\cite[Definitions 3.7 and 3.12]{GH19}),
building on the usual definition of the ring of motivic functions
attached to an $\Ldp$-definable set (see Definition \ref{def:motivic function}).
Any affine algebraic $\rats$-variety $X$ can be identified with
an $\Ldp$-definable set by choosing some $\ints$-model $\widetilde{X}$
of $X$. Roughly speaking, a motivic function $f$ on a $\rats$-variety
$X$ is a collection $\{f_{F}\}_{F\in\mathrm{Loc}_{>}}$ of functions
$f_{F}:X(F)\rightarrow\complex$, which is locally (on open affine
subsets) determined by a collection of $\Ldp$-formulas as in Definition
\ref{def:motivic function}. For a smooth $\rats$-variety $X$, a
collection of measures $\mu=\{\mu_{F}\}_{F\in\mathrm{Loc}_{>}}$ on
$\{X(F)\}_{F\in\mathrm{Loc}_{>}}$ is a \textit{motivic} measure on
$X$ if there exists an open affine cover $X=\bigcup\limits _{j=1}^{l}U_{j}$,
such that $\mu_{F}|_{U_{j}(F)}=(f_{j})_{F}\left|\omega_{j}\right|_{F}$,
where $f_{j}\in\mathcal{C}(U_{j})$ and $\omega_{j}$ is a non-vanishing
top differential form on $U_{j}$. An important property of the class of motivic functions is that it is preserved under integration (see \cite[Theorem 10.1.1]{CL08} and \cite[Theorem 4.3.1]{CGH14}).

The (FRS) property of a morphism $\varphi:X\to G$ has an equivalent
analytic characterization in terms of continuity of the pushforward
measures $\varphi_{*}(\mu_{F})$, where $\{\mu_{F}\}_{F\in\mathrm{Loc}_{>}}$
is a certain collection of smooth, compactly supported measures on
$\{X(F)\}_{F\in\mathrm{Loc}_{>}}$ (see Theorem \ref{Analytic condition for (FRS)}
and Proposition \ref{prop:reduction to an analytic}). Since the collection
$\{\mu_{F}\}$ can be chosen to be motivic, and since the pushforward
of a motivic measure is motivic, Theorem \ref{Main result} can be
reduced to statements about motivic functions. These statements are
Theorems \ref{Main model theoretic result for vector spaces} and
\ref{Main model theoretic result-Varieties}. \begin{thmx} \label{Main model theoretic result for vector spaces}
Let $h\in\mathcal{C}(\mathbb{A}_{\rats}^{n})$ be a motivic function
and assume that $h_{F}\in L^{1}(F^{n})$ for every $F\in\mathrm{Loc}_{>}$.
Then there exists $\epsilon>0$, such that $h_{F}\in L^{1+\epsilon}(F^{n})$
for every $F\in\mathrm{Loc}_{>}$. \end{thmx}
\begin{defn}
\label{def: smooth and locally Ls}
Let $X$ be an analytic variety over a non-Archimedean local field
$F$, with ring of integers $\mathcal{O}_{F}$. 
\begin{enumerate}
\item A measure $\mu$ on $X$ is called \textsl{smooth} if for every $x\in X$
there exists $x\in U\subseteq X$ and an analytic diffeomorphism $\psi:U\rightarrow\mathcal{O}_{F}^{\mathrm{dim}X}$,
such that $\psi_{*}(\mu)$ is a Haar measure on $\mathcal{O}_{F}^{\mathrm{dim}X}$. 
\item Let $f:X\rightarrow\mathbb{C}$ be a function on $X$ and $s\in\reals_{>0}$.
We say that $f$ is \textit{locally-$L^{s}$,} and write $f\in L_{\mathrm{Loc}}^{s}(X)$,
if for every open compact $U\subseteq X$, and for every (or equivalently
for some positive) smooth measure $\mu$ on $U$, we have $f\in L^{s}(U,\mu)$. 
\end{enumerate}
\end{defn}

\begin{thmx} \label{Main model theoretic result-Varieties}Let $X$
be a smooth $\rats$-variety and let $h\in\mathcal{C}(X)$. Assume
that $h_{F}\in L_{\mathrm{Loc}}^{1}(X(F))$ for every $F\in\mathrm{Loc}_{>}$.
Then there exists $\epsilon>0$ such that $h_{F}\in L_{\mathrm{Loc}}^{1+\epsilon}(X(F))$
for every $F\in\mathrm{Loc}_{>}$. \end{thmx}

Theorems \ref{Main model theoretic result for vector spaces} and
\ref{Main model theoretic result-Varieties} can be generalized to
statements about families of functions, which can be used to deduce
Theorem \ref{main result for families}. These are Theorems \ref{Main model theoretic result for families of vector spaces}
and \ref{Model theoretic result- families of varieties}: \begin{thmx}[Theorem
\ref{thm: Main model theoretic result for families of vector spaces v2}]
\label{Main model theoretic result for families of vector spaces}Let
$h\in\mathcal{C}(\mathbb{A}_{\rats}^{n}\times Y)$ be a family of
motivic functions parameterized by a $\rats$-variety $Y$, and assume
that for every $F\in\mathrm{Loc}_{>}$ we have $h_{F}|_{F^{n}\times\{y\}}\in L^{1}(F^{n})$
for every $y\in Y(F)$. Then there exists $\epsilon>0$, such that for
every $F\in\mathrm{Loc}_{>}$ we have $h_{F}|_{F^{n}\times\{y\}}\in L^{1+\epsilon}(F^{n})$,
for every $y\in Y(F)$. \end{thmx}

\begin{thmx} \label{Model theoretic result- families of varieties}Let
$\varphi:X\rightarrow Y$ be a smooth morphism of smooth algebraic
$\rats$-varieties. Let $h\in\mathcal{C}(X)$ be a motivic function,
and assume that for every $F\in\mathrm{Loc}_{>}$ we have $h_{F}|_{X_{y}(F)}\in L_{\mathrm{Loc}}^{1}(X_{y}(F))$
for every $y\in Y(F)$. Then there exists $\epsilon>0$, such that for
every $F\in\mathrm{Loc}_{>}$ we have $h_{F}|_{X_{y}(F)}\in L_{\mathrm{Loc}}^{1+\epsilon}(X_{y}(F))$
for every $y\in Y(F)$. \end{thmx}
\subsection{Further discussion of the main results}

In \cite{GH19} we proved Theorems \ref{Main result}, \ref{main result for families}
and \ref{FRS on complexity} in the case where $G$ is a vector space.
The proof of Theorem \ref{Main result} can be divided into four parts: 
\begin{enumerate}
\item Reduction to the case when $K=\rats$ (Proposition \ref{prop:reduction to Q},
cf. \cite[Section 6]{GH19}). 
\item Reduction to an analytic statement (Proposition \ref{prop:reduction to an analytic},
cf. \cite[Proposition 3.16]{GH19}). 
\item Reduction to a model theoretic statement: 
\begin{enumerate}
\item Reduction of (2) to Theorem \ref{Main model theoretic result-Varieties}. 
\item Further reduction to Theorem \ref{Main model theoretic result for vector spaces}. 
\end{enumerate}
\item Proof of Theorem \ref{Main model theoretic result for vector spaces}
(the stronger Theorem \ref{Main model theoretic result for families of vector spaces}
is proved in Section \ref{sec:Main-analytic-result}). 
\end{enumerate}
The proof of the first two parts is essentially the same as in \cite{GH19}.
Let $\mu=\{\mu_{F}\}_{F\in\mathrm{Loc}}$ be a motivic measure on
$X$ such that $\mu_{F}$ is smooth, non-negative and supported on
$X(\mathcal{O}_{F})$ for every $F\in\mathrm{Loc}$. Such a measure
exists by \cite[Proposition 3.14]{GH19}. The reduction to Proposition
\ref{prop:reduction to an analytic} implies that in order to deduce
Theorem \ref{Main result}, we need to find $N\in\mathbb{N}$, such
that for $F\in\mathrm{Loc}_{>}$ the measure $\varphi_{*}^{*N}(\mu_{F}\times\ldots\times\mu_{F})$
has continuous density with respect to the normalized Haar measure
on $G(\mathcal{O}_{F})$.

The difference between this paper and \cite{GH19} lies in (3) and
(4). In \cite{GH19}, the decay properties of the Fourier transform
of $\varphi_{*}(\mu_{F})$ were studied (\cite[Theorem 5.2]{GH19}),
and were used to deduce that after sufficiently many self-convolutions
we obtain a measure with continuous density (\cite[Corollary 5.3]{GH19}).
A key ingredient in the proof was the fact that the Fourier transform
is well behaved with respect to motivic functions\footnote{Furthermore, Cluckers and Loeser formulated the commutative Fourier
transform in a motivic language, and introduced a class of motivic
exponential functions, which is preserved under Fourier transform,
see \cite[Section 7]{CL10} and \cite[Section 3.4]{CH18}.}.

If one wishes to use the line of proof of \cite{GH19} in the general
case, then a non-commutative Fourier transform must be used and this
adds a serious complication. Due to this issue, we take a different
approach, showing that given a motivic measure $\sigma=\{\sigma_{F}\}_{F\in\mathrm{Loc}}$
on $G$, such that $\sigma_{F}$ is supported on $G(\mathcal{O}_{F})$
and has an $L^{1}$-density with respect to the normalized Haar measure
on $G(\mathcal{O}_{F})$, then there exists $\epsilon>0$ such that
$\sigma_{F}$ has $L^{1+\epsilon}$-density for every $F\in\mathrm{Loc}_{>}$.
This is Corollary \ref{Cor:Main model theoretic result for algebraic groups}
and it immediately follows from Theorem \ref{Main model theoretic result-Varieties}.
Taking $\sigma:=\varphi_{*}(\mu)$ for $\mu$ as above, and applying
Young's convolution inequality yields the existence of an $N(\epsilon)\in\nats$
such that $\varphi_{*}^{*N(\epsilon)}(\mu_{F}\times\ldots\times\mu_{F})$
has continuous density as required.

Since, locally, any smooth variety admits an \'etale morphism to
an affine space, and using the fact that \'etale morphisms preserve
the $L_{\mathrm{Loc}}^{1+\epsilon}$ property of functions on $F$-analytic
manifolds (see Lemma \ref{lem:etale preserves Lp}), it follows that
Theorem \ref{Main model theoretic result-Varieties} can be reduced
to an analogous claim about vector spaces, i.e.~Theorem \ref{Main model theoretic result for vector spaces}. 
\begin{rem}
Note that after the reduction to Theorem \ref{Main model theoretic result for vector spaces},
we are again in the realm of vector spaces, for which it is tempting
to use the results of \cite{GH19}. In other words,
a possible naive approach for proving Theorem \ref{Main model theoretic result for vector spaces}
is to use \cite[Theorem 5.2]{GH19} to show that 
the Fourier transform
$\left|\mathcal{F}(h_{F})(y)\right|$ of
functions $h$ as
in Theorem \ref{Main model theoretic result for vector spaces},	  decays faster than $\left|y\right|{}^{\alpha}$
for some $\alpha<0$, and then answer the following general analytic question: 
\begin{question}\label{que:L1+epsilon}Let
$F$ be a local field, and let $h$ be any compactly supported, $L^{1}$-function
on $F^{n}$ whose Fourier transform $\left|\mathcal{F}(h)(y)\right|$
decays faster than $\left|y\right|{}^{\alpha}$ for some $\alpha<0$.
Is there an $\epsilon>0$ such that $h$ is $L^{1+\epsilon}$?\end{question}

This is false, as an example given in Appendix \ref{Appendix: L1+epsilon} demonstrates.
In particular, we see that the above naive approach cannot work, and it is indeed necessary to restrict our attention to the class of motivic functions.
\end{rem}

\subsubsection{Discussion of the model-theoretic results}

In \cite{Igu74,Igu75}, it was shown that given a polynomial $h\in\ints_{p}[x_{1},...,x_{n}]$
and $0<s\in\reals$, then the Igusa zeta function 
\begin{equation}
Z_{h}(s,p):=\int_{\Zp^{n}}\left|h(x)\right|_{p}^{s}dx\label{eq:(1.1)}
\end{equation}
is a rational function in $p^{-s}$ for any $p$. In \cite{Den84,Pas89,Mac90,DL01},
variations of the above integral were studied, and the theorem on
the rationality of (\ref{eq:(1.1)}) was generalized for an integral
of the form 
\[
Z_{h,\psi}(s,p)=\int_{W_{p}(\psi)}\left|h(x)\right|_{p}^{s}dx,
\]
where $\psi$ is an $\mathcal{L}_{\mathrm{DP}}$-formula and $W_{p}(\psi)=\{x\in\Qp^{n}:\psi(x)\text{ holds}\}$.
In \cite{BDOP13}, integrals of the form 
\[
Z_{f,\psi}(s,F):=\int_{W_{F}(\psi)}\left|f_{F}(x)\right|_{F}^{s}dx,
\]
were investigated, where $f=\{f_{F}:F^{n}\rightarrow F\}_{F\in\mathrm{Loc}_{>}}$
is now an $\mathcal{L}_{\mathrm{DP}}$-definable function, $\psi$
is an $\mathcal{L}_{\mathrm{DP}}$-formula and $W_{F}(\psi)=\{x\in F^{n}:\psi(x)\text{ holds}\}$.  
In Theorems \ref{Main model theoretic result for vector spaces} and
\ref{Main model theoretic result for families of vector spaces} we
investigate integrals of the form 
\[
I_{h}(s,F):=\int_{F^{n}}\left|h_{F}(x)\right|^{s}dx,
\]
where $h=\{h_{F}:F^{n}\rightarrow\reals\}$ is an $\Ldp$-motivic
function (note that we take the usual absolute value $|~\cdot~|$ on
$\reals$). This is a generalization of the last case, as $Z_{f,\psi}(s,F)=I_{h}(s,F)$,
with $h_{F}=\left|f_{F}(x)\right|_{F}\cdot1_{W_{F}(\psi)}$ and $f$
an $\Ldp$-definable function as before.

We want to find $\epsilon>0$ such that if $I_{h}(1,F)<\infty$ (i.e.~$h_{F}$ is absolutely integrable) for $F\in\mathrm{Loc}_{>}$, then
$I_{h}(1+\epsilon,F)<\infty$ for $F\in\mathrm{Loc}_{>}$. For $h_{F}=\left|f_{F}(x)\right|_{F}\cdot1_{W_{F}(\psi)}$ where $f$ is $\Ldp$-definable, 
this can be deduced from \cite[Theorem B]{BDOP13}. For $h$ motivic
(Definition \ref{def:motivic function}), some complications arise;
$h$ is now a finite sum of terms $h=\sum\limits _{j=1}^{N}h_{i}$,
where each $h_{i}$ does not have to be absolutely integrable (at
least globally). In addition, each $h_{i}$ has a more complicated
description than a definable function. These complications are dealt
with in Section \ref{sec:Main-analytic-result}. The main idea in
both the definable and motivic cases is to reduce $I_{h}(s,F)$ to
a geometric power series (or a slight variant of such power series
in the motivic case), whose convergence is easier to analyze.

Let us explain the method for $h=\{\left|f_{F}(x)\right|_{F}\}_{F\in\Loc_{>}}$,
with $f$ definable. Let $q_{F}$ be the size of the residue field
$k_{F}$ of $F$. We can write $I_{h}(s,F)$ as a sum over the level
sets of $h_{F}$, that is $I_{h}(s,F)=\underset{k\in\ints}{\sum}\mu_{k,F}\cdot q_{F}^{-ks}$,
where $\mu_{k,F}$ is the measure of the level set $\{x\in F^{n}:\val(f_{F}(x))=k\}$.
The convergence of $I_{h}(s,F)$ then depends on the asymptotic behavior
of $\mu_{k,F}$ with respect to $k$. By analyzing $\mu_{k,F}$ it can be shown
(e.g. \cite[proof of Theorem B]{BDOP13} or \cite[Theorem 5.1]{Pas89})
that $I_{h}(s,F)$ can be written as a finite sum of expressions of
the form 
\begin{equation}
q_{F}^{-n}\sum_{\eta\in k_{F}^{r}}\sum_{\begin{array}{c}
l_{1},\ldots,l_{n},k\in\mathbb{Z}\\
\sigma(\eta,l_{1},\ldots,l_{n},k)
\end{array}}q_{F}^{-ks-l_{1}-\ldots-l_{n}},\label{eq:1.2}
\end{equation}
where $\sigma$ is an $\mathcal{L}_{\mathrm{DP}}$-formula. Using
elimination of quantifiers, and the rectilinearization Theorem (\cite[Theorem 2.1.9]{CGH14}),
we can write the expression appearing in (\ref{eq:1.2}) as a sum
of finitely many terms of the form 
\begin{equation}
\sum_{(e_{1},\ldots,e_{l})\in\mathbb{N}^{l}}q_{F}^{b_{1}(s)e_{1}+\ldots+b_{l}(s)e_{l}},\label{eq:1.3}
\end{equation}
where $b_{t}(s)$ are numbers depending on $s$. It can then be verified
that the set of $s\in\reals$ such that (\ref{eq:1.3}) is summable
is open, and does not depend on $F$, as required. In Section \ref{sec:Main-analytic-result}
we extend this result to the class of $\Ldp$-motivic functions, by
proving the more general Theorem \ref{Main model theoretic result for families of vector spaces}.

Integrability properties of motivic functions as above were furthermore studied in \cite{CGH14} and \cite{CGH18}, where the $L^p$-integrability locus for $p\in \{1,2,\infty\}$ was shown to be motivic (exponential),   giving rise to $L^p$-integrability transfer principles between mixed characteristic and positive characteristic local fields  (see \cite[Corollaries 3.2.2, 3.2.5]{CGH18}). We  expect that such transfer principles should still hold for any $p \in \mathbb{R}_{>0}$.
\subsection{{Future questions and directions}}

Two interesting questions can be asked following this work. Firstly,
one can try to measure how do the singularity properties of morphisms
improve under the convolution operation more finely, by considering
improvement of numeric singularity invariants. One especially interesting
candidate for such exploration is the minimal exponent ({see
\cite{Sai93,Saia,MPa,MPb,MPc}}). In the case of a divisor $E$ on
a smooth variety $X$, the minimal exponent $\tilde{\alpha}_{E}$
is defined as the negative of the largest root of $b_{E}(s)/(s+1)$,
where $b_{E}(s)$ is the Bernstein-Sato polynomial of $E$. {The
minimal exponent refines the log-canonical threshold $\mathrm{lct}(X,E)$,
which equals $\min\{1,\tilde{\alpha}_{E}\}$, and further detects
rational singularities - an effective integral divisor $E$ has rational
singularities if and only if $\tilde{\alpha}_{E}>1$ (\cite[Theorem 0.4]{Sai93}
and \cite{MPa}).} {In the case of maps into $\mathbb{A}^{1}$,
the minimal exponent has deterministic behavior with respect to
convolution; by a Thom-Sebastiani type result, if $0\neq f_{1},f_{2}\in\mathbb{C}[x_{1},...,x_{n}]$ are homogeneous polynomials,  
then $\tilde{\alpha}_{f_{1}*f_{2}}=\tilde{\alpha}_{f_{1}}+\tilde{\alpha}_{f_{2}}$
(see \cite{Sai94} and \cite[Example 6.8]{MPa}), where $\tilde{\alpha}_{f_{i}}=\tilde{\alpha}_{E_{i}}$
for $E_{i}$ the divisor defined by $f_{i}$.} The minimal exponent
can be defined for an arbitrary closed subscheme $Z\subset X$ ({See \cite{MPc}} and \cite[Theorem 4]{BMS06}). 
It would {be very interesting} to study the behavior of
the minimal exponent of the fibers of a dominant morphism $\varphi:X\to G$
under the convolution operation.

Secondly, one might try to generalize the main model theoretic results
of this paper. For example, one might ask whether Theorem \ref{Main model theoretic result for vector spaces}
holds when considering the more general class of motivic exponential
functions (e.g.~see \cite[Definition 2.7]{CGH16} or \cite{CL10}).

\subsection{Structure of the paper}

In Section \ref{sec:preliminaries} we recall relevant preliminary
material. In Section \ref{sec:Properties of convolutions of morphisms}
we prove Theorem \ref{thm: singularity properties obtained after convolution}.
In Section \ref{sec:Main-analytic-result} we prove Theorems \ref{Main model theoretic result for vector spaces},
\ref{Main model theoretic result-Varieties}, \ref{Main model theoretic result for families of vector spaces}
and \ref{Model theoretic result- families of varieties}. In Section
\ref{sec:Proof-of-the main result} we prove Theorems \ref{Main result},
\ref{main result for families} and \ref{FRS on complexity}. In Appendix
\ref{Appendix: L1+epsilon} we provide a counter example to  Question \ref{que:L1+epsilon}.

\subsection{Conventions}

Throughout the paper we use the following conventions: 
\begin{itemize}
\item Unless explicitly stated otherwise, $K$ is a field of characteristic
$0$ and $F$ is a non-Archimedean local field whose ring of integers
is $\mathcal{O}_{F}$. 
\item For any $K$-scheme $X$ and $x\in X$ we denote by $\kappa(\{x\})$ its
residue field. 
\item For a morphism $\varphi:X\rightarrow Y$ of $K$-schemes, the scheme
theoretic fiber at $y\in Y$ is denoted by either $X_{y,\varphi}$
or $\spec(\kappa(\{y\}))\times_{Y}X$ (if $\varphi$ is understood we omit it from our notation). 
\item For a field extension $K'/K$ and a $K$-variety $X$ (resp $K$-morphism
$\varphi:X\rightarrow Y$), we denote the base change of $X$ (resp.
$\varphi$) by $X_{K'}:=X\times_{\spec(K)}\spec(K')$ (resp. $\varphi_{K'}:X_{K'}\rightarrow Y_{K'}$).
\item For a $K$-morphism $\varphi:X\rightarrow Y$ between $K$-varieties
$X$ and $Y$, we denote by $X^{\mathrm{sm}}$ (resp. $X^{\mathrm{ns}}$)
the smooth (resp. non-smooth) locus of $X$, and by $X^{\mathrm{sm,\varphi}}$
(resp. $X^{\mathrm{ns,\varphi}}$) the smooth (resp. non-smooth) locus
of $\varphi$ in $X$. 
\item We use $F\in\mathrm{Loc}_{>}$ to denote ``$F\in\mathrm{Loc}$ with
large enough residual characteristic''. 
\end{itemize}

\subsection{Acknowledgements}

We thank Raf Cluckers, Ehud Hrushovski, Moshe Kamenski, Gady Kozma
and Dan Mikulincer for useful conversations. We thank Shai Shechter
for both useful conversations and for reading a preliminary version
of this paper. We also thank the anonymous referee for valuable comments and remarks. 
A large part of this work was carried out while visiting
Nir Avni at the mathematics department at Northwestern University,
we thank them and Nir for their hospitality. We also wish to thank
Nir for many helpful discussions, and for raising the question this
paper answers together with Rami Aizenbud. Finally, it is a pleasure
to thank our advisor Rami Aizenbud for numerous useful conversation,
for his guidance, and for suggesting this question together with Nir.

Both authors where partially supported by ISF grant 687/13, BSF grant
2012247 and a Minerva foundation grant.

\section{Preliminaries \label{sec:preliminaries}}

\subsection{Non-commutative Fourier transform}

In this subsection we follow \cite[Sections 2.3, 4.1, 4.2]{App14}.
Let $G$ be a compact Hausdorff second countable group and let $\hat{G}$
be the set of equivalence classes of irreducible representations of
$G$. Define the set $\mathcal{M}(\hat{G}):=\bigcup_{\pi\in\hat{G}}\mathrm{End}_{\complex}(\pi)$.
We say that a map $T:\hat{G}\rightarrow\mathcal{M}(\hat{G})$ is \textit{compatible}
if $T(\pi)\in\mathrm{End}_{\complex}(\pi)$ for any $\pi\in\hat{G}$.
We denote the space of compatible mappings by $\mathcal{L}(\hat{G})$.
The \textit{non-commutative Fourier transform} is the map $\mathcal{F}:L^{1}(G)\rightarrow\mathcal{L}(\hat{G})$,
defined by 
\[
\mathcal{F}(f)(\pi)=\int_{G}f(g)\cdot\pi(g^{-1})dg,
\]
for each $\pi\in\hat{G}$, where $dg$ is the normalized Haar measure.
For $1\leq p<\infty$, we set $\mathcal{H}_{p}(\hat{G})$ to be the
linear space of all $T\in\mathcal{L}(\hat{G})$ for which 
\[
\left\Vert T\right\Vert _{p}:=\left(\sum_{\pi\in\hat{G}}\mathrm{dim}(\pi)\cdot\left\Vert T(\pi)\right\Vert _{\mathrm{Sch},p}^{p}\right)^{\frac{1}{p}}<\infty,
\]
where $\left\Vert T(\pi)\right\Vert _{\mathrm{Sch},p}:=\left(\mathrm{trace}(\left(T(\pi)T(\pi)^{*}\right)^{p/2})\right)^{\frac{1}{p}}$
is the Schatten $p$-norm. This gives $\mathcal{H}_{p}(\hat{G})$
a structure of a Banach space. In particular, $\left\Vert T\right\Vert _{2}^{2}=\underset{\pi\in\hat{G}}{\sum}\mathrm{dim}(\pi)\cdot\left\Vert T(\pi)\right\Vert _{\mathrm{HS}}^{2}<\infty$,
where $\|\cdot\|_{\mathrm{HS}}$ is the Hilbert-Schmidt norm. This
gives $\mathcal{H}_{2}(\hat{G})$ a structure of a complex Hilbert
space with an inner product 
\[
\langle T_{1},T_{2}\rangle:=\sum_{\pi\in\hat{G}}\mathrm{dim}(\pi)\cdot\langle T_{1}(\pi),T_{2}(\pi)\rangle_{\mathrm{HS}}.
\]
The restriction of $\mathcal{F}$ to $L^{2}(G)$ has the following
properties: 
\begin{thm}[{{See e.g. \cite[Theorem 2.3.1]{App14}}}]
~\label{Proposition Fourier L2} 
\begin{enumerate}
\item (Fourier expansion) For all $f\in L^{2}(G)$, we have 
\[
f(g)=\sum_{\pi\in\hat{G}}\mathrm{dim}(\pi)\cdot\mathrm{trace}(\mathcal{F}(f)(\pi)\pi(g))
\]
\item (Parseval-Plancherel identity) The operator $\mathcal{F}$ is an isometry
from $L^{2}(G)$ into $\mathcal{H}_{2}(\hat{G})$ so that for all
$f,f_{1},f_{2}\in L^{2}(G)$, 
\[
\int_{G}\left|f(g)\right|^{2}dg=\sum_{\pi\in\hat{G}}\mathrm{dim}(\pi)\left\Vert \mathcal{F}(f)(\pi)\right\Vert _{\mathrm{HS}}^{2},
\]
and 
\[
\int_{G}f_{1}(g)\cdot\overline{f_{2}(g)}dg=\sum_{\pi\in\hat{G}}\mathrm{dim}(\pi)\langle\mathcal{F}(f_{1})(\pi),\mathcal{F}(f_{2})(\pi)\rangle_{\mathrm{HS}}.
\]
\end{enumerate}
\end{thm}

Here are some additional properties of the Fourier transform: 
\begin{thm}[{{\cite[Theorem 2.3.2]{App14} and \cite[Section 2.14]{Edw72}}}]
~\label{Fourier transform of functions} 
\begin{enumerate}
\item If $1\leq p\leq2$ and $f\in L^{p}(G)$, then $\mathcal{F}(f)\in\mathcal{H}_{q}(\hat{G})$,
where $\frac{1}{p}+\frac{1}{q}=1$ and $\left\Vert \mathcal{F}(f)\right\Vert _{q}\leq\left\Vert f\right\Vert _{p}$. 
\item Let $\mathcal{A}(G):=\{f\in L^{1}(G):\left\Vert \mathcal{F}(f)\right\Vert _{1}<\infty\}$.
Then $\mathcal{A}(G)$ consists of continuous functions, and is a
commutative Banach algebra with respect to convolution $f_{1}*f_{2}(x)=\int_{G}f_{1}(g)f_{2}(g^{-1}x)dg$. 
\end{enumerate}
\end{thm}

The Fourier transform can be defined for probability measures as well.
For a probability measure $\mu$ and any $\pi\in\hat{G}$ we define
\[
\mathcal{F}(\mu)(\pi)(v):=\int_{G}\pi(g^{-1})vd\mu.
\]
Notice that if $\mu$ is absolutely continuous with respect to $dg$
with density $f_{\mu}$, then $\mathcal{F}(\mu)(\pi)=\mathcal{F}(f_{\mu})(\pi)$. 
\begin{prop}
\label{prop convolution+Fourier} Let $\mu_{1}$ and $\mu_{2}$ be
probability measures on $G$, and let $\pi\in\hat{G}$. Then 
\[
\mathcal{F}(\mu_{1}*\mu_{2})(\pi)=\mathcal{F}(\mu_{1})(\pi)\cdot\mathcal{F}(\mu_{2})(\pi).
\]
\end{prop}

Finally, the spaces $\mathcal{H}_{p}(\hat{G})$ satisfy the classical
H\"older's inequality, as well as its generalization: 
\begin{prop}[{Generalization of H\"older's inequality}]
\label{prop:(Generalization-of-Holder} Let $r\in(0,\infty]$ and
let $p_{1},{\ldots},p_{n}\in(0,\infty]$ such that $\sum\limits _{k=1}^{n}\frac{1}{p_{k}}=\frac{1}{r}$.
Then for any collection $\{T_{k}\}_{k=1}^{n}$, with $T_{k}\in\mathcal{H}_{p_{k}}(\hat{G})$
we have $\prod_{k=1}^{n}T_{k}\in\mathcal{H}_{r}(\hat{G})$. 
\end{prop}

\begin{proof}
The Schatten norms satisfy (a generalized version of) H\"older's inequality,
that is, for any $A_{1},A_{2},...,A_{n}\in\mathrm{End}_{\complex}(\pi)$
we have $\left\Vert \prod_{k=1}^{n}A_{k}\right\Vert _{\mathrm{Sch},r}\leq\prod_{k=1}^{n}\left\Vert A_{k}\right\Vert _{\mathrm{Sch},p_{k}}$.
Hence, 
\begin{align*}
\left\Vert \prod_{k=1}^{n}T_{k}\right\Vert _{r} & =\left(\sum_{\pi\in\hat{G}}\mathrm{dim}(\pi)\cdot\left\Vert \prod_{k=1}^{n}T_{k}(\pi)\right\Vert _{\mathrm{Sch},r}^{r}\right)^{\frac{1}{r}}\\
 & \leq\left(\sum_{\pi\in\hat{G}}\mathrm{dim}(\pi)\cdot\left(\prod_{k=1}^{n}\left\Vert T_{k}(\pi)\right\Vert _{\mathrm{Sch},p_{k}}\right)^{r}\right)^{\frac{1}{r}}\\
 & \leq\prod_{k=1}^{n}\left(\sum_{\pi\in\hat{G}}\mathrm{dim}(\pi)\cdot\left\Vert T_{k}(\pi)\right\Vert _{\mathrm{Sch},p_{k}}^{p_{k}}\right)^{1/p_{k}}=\prod_{k=1}^{n}\left\Vert T_{k}\right\Vert _{{p_{k}}}<\infty,
\end{align*}
where the second inequality follows from the generalized H\"older inequality
for $L^{p}(\hat{G},v)$, with respect to the measure $v(A)=\underset{\pi\in A}{\sum}\mathrm{dim}(\pi)$
for $A\subseteq\hat{G}$ (instead of the usual counting measure). 
\end{proof}
\begin{lem}
\label{lem: continuous function after enough convolutions} Let $G$
be a compact group and let $f\in L^{s}(G)$ for $1<s<\infty$. Then
there exists $N(s)\in\nats$ such that the $N(s)$-th convolution
power $f^{*N(s)}$ of $f$ is continuous. 
\end{lem}

\begin{proof}
By Young's convolution inequality, there exists $M(s)$ such that
$f^{*M(s)}\in L^{2}(G)$. Thus $\mathcal{F}(f^{*M(s)})\in\mathcal{H}_{2}(\hat{G})$.
Now, for $N(s)=2M(s)$ we have by Proposition \ref{prop convolution+Fourier}
\[
\mathcal{F}(f^{*N(s)})=\mathcal{F}(f^{*M(s)})\cdot\mathcal{F}(f^{*M(s)})\in\mathcal{H}_{1}(\hat{G}),
\]
and hence by Theorem \ref{Fourier transform of functions} the function
$f^{*N(s)}$ is continuous. 
\end{proof}

\subsection{\label{subsec:The-Denef-Pas-language,}The Presburger language, the
Denef-Pas language, and motivic functions }

The Presburger language, denoted 
\[
\mathcal{L}_{\mathrm{Pres}}=(+,-,\leq,\{\equiv_{\mathrm{mod}~n}\}_{n>0},0,1)
\]
consists of the language of ordered abelian groups along with constants
$0,1$ and a family of $2$-relations $\{\equiv_{\mathrm{mod}~n}\}_{n>0}$
of congruences modulo $n$. We consider in this paper only the structures
isomorphic to $\ints$. 
\begin{defn}[{{See \cite[Definition 1]{Clu03} and \cite[Section 4.1]{CL08}}}]
\label{def:-Linear function} Let $S\subseteq\ints^{n}$ and $X\subset S\times\ints^{m}$
be $\mathcal{L}_{\mathrm{Pres}}$-definable sets. We call a definable
function $f:X\to\ints$ \textit{$S$-linear} if there is an $\mathcal{L}_{\mathrm{Pres}}$-definable
function $\gamma:S\rightarrow\ints$ and integers $a_{i}$ and $0\leq c_{i}<n_{i}$
for $i=1,...,m$ such that $x_{i}-c_{i}\equiv0\text{ }\mathrm{mod}\text{ }n_{i}$
and $f(s,x_{1},...,x_{m})=\sum\limits _{i=1}^{m}a_{i}(\frac{x_{i}-c_{i}}{n_{i}})+\gamma(s)$.
If $S$ is a point (and hence $\gamma$ is a constant), we say that
$f$ is \textit{linear}. 
\end{defn}

\begin{thm}[{{Presburger cell decomposition \cite[Theorem 1]{Clu03}}}]
\label{Presburger Cell decomposition} Let $S\subseteq\ints^{n}$,
$X\subset S\times\ints^{m}$ and $f:X\to\ints$ be $\mathcal{L}_{\mathrm{Pres}}$-definable.
Then there exists a finite partition $\mathcal{P}$ of $X$ into $S$-cells
(see \cite[Definition 4.3.1]{CL08}), such that the restriction $f|_{A}:A\to\ints$
is $S$-linear for each cell $A\in\mathcal{P}$. 
\end{thm}

The Denef-Pas language $\Ldp=(\mathcal{L}_{\mathrm{Val}},\mathcal{L}_{\mathrm{Res}},\mathcal{L}_{\mathrm{Pres}},\text{\ensuremath{\val}, \ensuremath{\ac}})$
is a first order language with three sorts of variables: 
\begin{itemize}
\item The valued field sort $\VF$ endowed with the language of rings $\mathcal{L}_{\mathrm{Val}}=(+,-,\cdot,0,1)$. 
\item The residue field sort $\RF$ endowed with the language of rings $\mathcal{L}_{\mathrm{Res}}=(+,-,\cdot,0,1)$. 
\item The value group sort $\VG$ (which we just call $\ints$), endowed
with the Presburger language $\mathcal{L}_{\mathrm{Pres}}=(+,-,\leq,\{\equiv_{\mathrm{mod}~n}\}_{n>0},0,1)$. 
\item A function $\val:\VF\backslash\{0\}\rightarrow\VG$ for a valuation
map. 
\item A function $\ac:\VF\rightarrow\RF$ for an angular component map. 
\end{itemize}
Let $\mathrm{Loc}$ be the collection of all non-Archimedean local
fields and $\mathrm{Loc}_{M}$ be the set of $F\in\mathrm{Loc}$ such
that $F$ has residue field $k_{F}$ of characteristic larger than
$M$. We will use the notation $F\in\mathrm{Loc}_{>}$ to denote ``$F\in\mathrm{Loc}$
with large enough residual characteristic''. For any $F\in\mathrm{Loc}$
and choice of a uniformizer $\pi$ of $\mathcal{O}_{F}$, the pair
$(F,\pi)$ is naturally a structure of $\Ldp$. Since our results
are independent of the choice of a uniformizer, we omit it from our
notations. Therefore, given a formula $\phi$ in ${\Ldp}$, with $n_{1}$
free valued field variables, $n_{2}$ free residue field variables
and $n_{3}$ free value group variables, we can naturally interpret
it in $F\in\mathrm{Loc}$, yielding a subset $\phi(F)\subseteq F^{n_{1}}\times k_{F}^{n_{2}}\times\mathbb{Z}^{n_{3}}$. 
\begin{defn}[{{{See \cite[Definitions 2.3-2.6]{CGH16}}}}]
\label{def:motivic function} Let $n_{1},n_{2},n_{3}$ and $M$ be
natural numbers. 
\begin{enumerate}
\item A collection $X=(X_{F})_{F\in\mathrm{Loc}_{M}}$ of subsets $X_{F}\subseteq F^{n_{1}}\times k_{F}^{n_{2}}\times\mathbb{Z}^{n_{3}}$
is called a \textit{definable set} if there is an ${\Ldp}$-formula
$\phi$ and $M'\in\nats$ such that $X_{F}=\phi(F)$ for every $F\in\mathrm{Loc}_{M'}$. 
\item Let $X$ and $Y$ be definable sets. A \textit{definable function}
is a collection $f=(f_{F})_{F\in\mathrm{Loc}_{M}}$ of functions $f_{F}:X_{F}\rightarrow Y_{F}$,
such that the collection of their graphs $\{\Gamma_{f_{F}}\}_{F\in\mathrm{Loc}_{M}}$
is a definable set. 
\item Let $X$ be a definable set. A collection $h=(h_{F})_{F\in\mathrm{Loc}_{M}}$
of functions $h_{F}:X_{F}\rightarrow\mathbb{R}$ is called a \textit{motivic}
(or \textit{constructible}) function on $X$, if for $F\in\mathrm{Loc}_{M}$
it can be written in the following way (for every $x\in X_{F}$):
\[
h_{F}(x)=\sum_{i=1}^{N}|Y_{i,F,x}|q_{F}^{\alpha_{i,F}(x)}\left(\prod_{j=1}^{N'}\beta_{ij,F}(x)\right)\left(\prod_{j=1}^{N''}\frac{1}{1-q_{F}^{a_{ij}}}\right),
\]
where, 
\begin{itemize}
\item $N,N'$ and $N''$ are integers and $a_{il}$ are non-zero integers. 
\item $\alpha_{i}:X\rightarrow\mathbb{Z}$ and $\beta_{ij}:X\rightarrow\mathbb{Z}$
are definable functions. 
\item $Y_{i,F,x}=\{\xi\in k_{F}^{r_{i}}:(x,\xi)\in{Y}_{i,F}\}$ is the fiber
over $x$ where $Y_{i}\subseteq X\times\mathrm{RF}^{r_{i}}$ are definable
sets and $r_{i}\in\nats$. 
\item The integer $q_{F}$ is the size of the residue field $k_{F}$. 
\end{itemize}
\end{enumerate}
The set of motivic functions on a definable set $X$ forms a ring,
which we denote by $\mathcal{C}(X)$. 
\end{defn}

{The following theorem is used} in the proof of the
main analytic result, Theorem \ref{Main model theoretic result for families of vector spaces}. 
\begin{thm}[{Uniform rectilinearization, see \cite[Theorem 4.5.4]{CGH14}}]
\label{thm:-(Uniform Rectilinearization)} Let $Y$ and $X\subseteq Y\times\ints^{m}$
be $\Ldp$-definable sets. Then there exist finitely many $\Ldp$-definable
sets $A_{i}\subset Y\times\ints^{m}$ and $B_{i}\subset Y\times\ints^{m}$
and $\Ldp$-definable isomorphisms $\rho_{i}:A_{i}\to B_{i}$ over
$Y$ such that for $F\in\mathrm{Loc}_{>}$ the following hold: 
\begin{enumerate}
\item The sets $A_{i,F}$ are disjoint and their union equals $X_{F}$, 
\item For every $i$, the function $\rho_{i,F}$ can be written as 
\[
\rho_{i,F}(y,x_{1},...,x_{m})=(y,\alpha_{i,F}(x_{1},...,x_{m})+\beta_{i,F}(y)),
\]
with $\alpha_{i}$ being $\mathcal{L}_{\mathrm{Pres}}$-linear, and
$\beta_{i}$ an $\Ldp$-definable function from $Y$ to $\ints$. 
\item For each $y\in Y_{F}$, the set $B_{i,F,y}$ is a set of the form
$\Lambda_{y}\times\ints_{\geq0}^{l_{i}}$ for a finite set $\Lambda_{y}\subset\ints_{\geq0}^{m-l_{i}}$
depending on $y$, with an integer $l_{i}\geq0$ depending only on
$i$. 
\end{enumerate}
\end{thm}

\subsection{The (FRS) property}

Recall that given a variety $X$, a resolution of singularities is
a proper birational map $p:\widetilde{X}\to X$ from a smooth variety
$\widetilde{X}$ to $X$. 
\begin{defn}
\label{defn: rational sings} We say that $X$ has \textit{rational
singularities} if for any resolution of singularities $p:\widetilde{X}\to X$
the natural morphism $\mathcal{O}_{X}\to Rp_{*}(\mathcal{O}_{\widetilde{X}})$
is a quasi-isomorphism where $Rp_{*}(-)$ is the higher direct image
functor. 
\end{defn}

\begin{defn}[{{{\cite[Section 1.2.1, Definition II]{AA16}}}}]
\label{def:FRS} Let $\varphi:X\to Y$ be a morphism between smooth
$K$-varieties $X$ and $Y$. 
\begin{enumerate}
\item We say that $\varphi:X\to Y$ is (FRS) at $x\in X(K)$ if it is flat
at $x$, and there exists an open $x\in U\subseteq X$ such that $U\times_{Y}\{\varphi(x)\}$
is reduced and has rational singularities. 
\item We say that $\varphi:X\to Y$ is (FRS) if it is flat and it is (FRS)
at $x$ for all $x\in X(\overline{K})$. 
\end{enumerate}
\end{defn}

The (FRS) property has the following analytic characterization: 
\begin{thm}[{{{\cite[Theorem 3.4]{AA16}}}}]
\label{Analytic condition for (FRS)} Let $\varphi:X\rightarrow Y$
be a map between smooth algebraic varieties defined over a finitely
generated field $K$ of characteristic $0$, and let $x\in X(K)$.
Then the following conditions are equivalent: 
\begin{enumerate}
\item $\varphi$ is (FRS) at $x$. 
\item There exists a Zariski open neighborhood $x\in U\subseteq X$ such
that for any $K\subseteq F\in\mathrm{Loc}$ and any smooth, compactly
supported measure $\mu$ on $U(F)$, the measure $(\varphi|_{U(F)})_{*}(\mu)$
has continuous density. 
\item For any finite extension $K'/K$, there exists $K'\subseteq F\in\mathrm{Loc}$
and a non-negative smooth, compactly supported measure $\mu$ on $X(F)$
that does not vanish at $x$ such that $(\varphi|_{X(F)})_{*}(\mu)$
has continuous density. 
\end{enumerate}
\end{thm}

\section{Properties of convolutions of morphisms \label{sec:Properties of convolutions of morphisms} }

In this section we discuss properties of the convolution operation
(as defined in Definition \ref{def:convolution}). We first recall
the following proposition from \cite{GH19}, which is a consequence
of Proposition \ref{prop:convpreservesgoodproperties}. 
\begin{prop}[{{{\cite[Corollary 3.3]{GH19}}}}]
\label{prop: properties preserved under convolution} Let $X$ and
$Y$ be smooth algebraic $K$-varieties, $G$ be an algebraic $K$-group
and let $S$ be any of the following properties of morphisms: 
\begin{enumerate}
\item Smoothness. 
\item (FRS). 
\item Flatness (or flatness with reduced fibers). 
\item Dominance. 
\end{enumerate}
Suppose $\varphi:X\to G$ has property $S$ and $\psi:Y\to G$ is
any morphism, then the maps $\varphi*\psi:X\times Y\to G$ and $\psi*\varphi:Y\times X\to G$
have the property $S$. 
\end{prop}

\begin{rem}
Dominance is not preserved under base change, but is still preserved
under the convolution operation. 
\end{rem}

\begin{defn}
\label{def:normality and (FGI)} Let $\varphi:X\to Y$ be a morphism
between $K$-schemes. 
\begin{enumerate}
\item The morphism $\varphi$ is called \textit{normal at $x\in X$} if
it is flat at $x$ and the fiber $X_{\varphi(x),\varphi}$ is geometrically
normal at $x$ over $\kappa(\{\varphi(x)\})$. Likewise, $\varphi$ is
called \textit{normal }if it is normal at every $x\in X$ (see \cite[Definition 36.18.1]{Sta}). 
\item The morphism $\varphi$ is called \textit{(FGI)} if it is flat with
geometrically irreducible fibers. 
\end{enumerate}
\end{defn}

\begin{lem}
The following properties of morphisms are preserved under base change
and composition: 
\begin{enumerate}
\item Normality between smooth $K$-varieties. 
\item (FGI). 
\end{enumerate}
\end{lem}

\begin{proof}
We start by showing that these properties are preserved under base
change; given a normal morphism $\varphi:X\rightarrow Y$, and a base
change $\widetilde{\varphi}:X\times_{Y}Z\rightarrow Z$ of $\varphi$
with respect to a morphism $\psi:Z\rightarrow Y$, the fibers of $\widetilde{\varphi}$
are base change of the fibers of $\varphi$ by a field extension,
and hence they are geometrically normal (see e.g. \cite[Lemma 32.10.4]{Sta}).
The proof for the (FGI) property is similar.

Let $\varphi_{1}:X\rightarrow Y$ and $\varphi_{2}:Y\rightarrow Z$
be two normal morphisms between smooth $K$-varieties. By Serre's
criterion $(S_{2}+R_{1})$ for normality (see e.g. \cite[Lemmas 10.151.4]{Sta})
and the fact that fibers of flat morphisms between smooth varieties
are Cohen-Macaulay (as they are local complete intersections), it
is enough to show that the fibers of $\varphi_{2}\circ\varphi_{1}$
are regular in codimension $1$. Let $W$ be a codimension one subvariety
of $X_{z,\varphi_{2}\circ\varphi_{1}}$, the fiber of $\varphi_{2}\circ\varphi_{1}$
at $z\in Z$. By flatness of $\varphi_{1}$, the set $\varphi_{1}(W)$
is of codimension at most one in $Y_{z,\varphi_{2}}$, and hence there
exists $y\in Y_{z,\varphi_{2}}$ such that $X_{y,\varphi_{1}}\cap W$
is not trivial, and $y$ is a smooth point of $Y_{z,\varphi_{2}}$,
or equivalently, $y$ is a smooth point of $\varphi_{2}$. But $X_{y,\varphi_{1}}\cap W$
is of codimension at most one in $X_{y,\varphi_{1}}$ and hence, by
assumption, there exists $x\in X_{y,\varphi_{1}}\cap W$ such that
$\varphi_{1}$ is smooth at $x$. This implies that $\varphi_{2}\circ\varphi_{1}$
is smooth at $x$ and hence $x$ is a smooth point of $X_{z,\varphi_{2}\circ\varphi_{1}}$,
so $W\cap X_{z,\varphi_{2}\circ\varphi_{1}}^{\mathrm{sm}}$ is not
empty as required.

Now let $\varphi_{1}:X\rightarrow Y$ and $\varphi_{2}:Y\rightarrow Z$
be two (FGI) morphisms. In order to prove that $\varphi_{2}\circ\varphi_{1}$
is (FGI) we need the following lemma: 
\begin{lem}[{{{See \cite[Lemma 5.8.12]{Sta}}}}]
\label{Lemma 3.3-Auxilary} Let $f:X_{1}\rightarrow X_{2}$ be a
continuous map between topological spaces, such that $f$ is open,
$X_{2}$ is irreducible, and there exists a dense collection of points
$y\in X_{2}$ such that $f^{-1}(y)$ is irreducible. Then $X_{1}$
is irreducible. 
\end{lem}

Now let $z\in Z$, set $K'=\overline{\kappa(\{z\})}$ and denote $X_{2}:=(Y_{z,\varphi_{2}})_{K'}$,
$X_{1}:=\left(X_{z,\varphi_{2}\circ\varphi_{1}}\right)_{K'}$, and
$f:=\varphi_{1}|_{X_{1}}:X_{1}\rightarrow X_{2}$. Notice that $f$
is flat as a base change of a flat map, and thus open. Moreover, by
our assumption, $X_{2}$ is irreducible and all the fibers of $f$
are irreducible. Since $f$ satisfies the conditions of the above
lemma, we deduce that $X_{1}$ is irreducible, as required. 
\end{proof}
As a consequence, we arrive at the following: 
\begin{cor}
\label{Cor 3.5} Proposition \ref{prop: properties preserved under convolution}
holds if the property $S$ is either normality or (FGI). 
\end{cor}

\begin{prop}
\label{prop: upper bounds for properties} Let $m\in\nats$, let $X_{1},\ldots,X_{m}$
be smooth $K$-varieties, let $G$ be a connected algebraic $K$-group,
and let $\{\varphi_{i}:X_{i}\rightarrow G\}_{i=1}^{m}$ be a collection
of strongly dominant morphisms. 
\begin{enumerate}
\item For any $1\leq i,j\leq m$ the morphism $\varphi_{i}*\varphi_{j}$
is surjective. 
\item We have $(X_{1}\times\ldots\times X_{m})^{\mathrm{ns},\varphi_{1}*\ldots*\varphi_{m}}\subseteq X_{1}^{\mathrm{ns},\varphi_{1}}\times\ldots\times X_{m}^{\mathrm{ns},\varphi_{m}}$,
and in particular the non-smooth locus of $\varphi_{1}*\ldots*\varphi_{m}$
is of codimension at least $m$ in $X_{1}\times\ldots\times X_{m}$. 
\item If $m\geq\mathrm{dim}G$ then $\varphi_{1}*\ldots*\varphi_{m}$ is
flat. 
\item If $m\geq\mathrm{dim}G+1$ then $\varphi_{1}*\ldots*\varphi_{m}$
is flat with reduced fibers. 
\item If $m\geq\mathrm{dim}G+2$ then $\varphi_{1}*\ldots*\varphi_{m}$
is flat with normal fibers (i.e.~it is a normal morphism). 
\item If $m\geq\mathrm{dim}G+k$, with $k>2$, then $\varphi_{1}*\ldots*\varphi_{m}$
is flat with normal fibers which are regular in codimension $k-1$. 
\end{enumerate}
\end{prop}

\begin{proof}
~ 
\begin{enumerate}
\item Follows from the fact that every two open dense sets $U_{1},U_{2}\subseteq G$
satisfy $U_{1}\cdot U_{2}=G$. 
\item This holds since smoothness is preserved under convolution and $X_{i}^{\mathrm{sm},\varphi_{i}}$
is of codimension at least $1$ for any $1\leq i\leq m$. 
\item It is enough to show that every fiber of $\varphi_{1}*...*\varphi_{m}$
is of co-dimension $\mathrm{dim}G$ (cf. \cite[III, Exercise 10.9]{Har77}).
Since $(X_{1}\times\ldots\times X_{m})^{\mathrm{ns},\varphi_{1}*\ldots*\varphi_{m}}$
is of codimension at least $m$ and the irreducible components of
any fiber of $\varphi_{1}*...*\varphi_{m}$ are of codimension at
most $\dim G$, it is sufficient to choose $m\geq\mathrm{dim}G$ to
guarantee that any irreducible component of a given fiber of $\varphi_{1}*...*\varphi_{m}$
contains a smooth point of $\varphi_{1}*...*\varphi_{m}$ and hence
is of codimension $\mathrm{dim}G$. 
\item Let $m\geq\mathrm{dim}G+1$ and let $Z$ be a fiber of $\varphi_{1}*...*\varphi_{m}$.
By (2) and (3) it follows that $(X_{1}\times\ldots\times X_{m})^{\mathrm{ns},\varphi_{1}*\ldots*\varphi_{m}}\cap Z$
is of codimension at least $m-\mathrm{dim}G$ in $Z$. In particular
$Z$ is generically reduced (by e.g. \cite[III, Theorem 10.2]{Har77}),
and since $\varphi_{1}*\ldots*\varphi_{m}$ is flat, it follows that
$Z$ is reduced as well (see e.g. \cite[Lemma 10.151.3]{Sta}). 
\item Similar to the proof of (4), where we now use Serre's criterion $(S_{2}+R_{1})$
for normality. 
\item Follows from the proof of (4). 
\end{enumerate}
\end{proof}
\begin{prop}
\label{prop:The-bounds-are tight} The bounds in Proposition \ref{prop: upper bounds for properties}
are tight. 
\end{prop}

\begin{proof}
Let $G=(\mathbb{A}^{m},+)$ and consider the map $\varphi:\mathbb{A}^{m}\rightarrow G$,
defined by 
\[
\varphi(x_{1},\ldots,x_{m})=(x_{1}^{2},\left(x_{1}x_{2}\right)^{2},\left(x_{1}x_{3}\right)^{2},\ldots,\left(x_{1}x_{m}\right)^{2}).
\]
Notice that the fibers of $\varphi$ are zero dimensional except the
fiber over $0$, which is $(m-1)$-dimensional. Now, $(\varphi^{*d})^{-1}(0)$
contains the $d(m-1)$-dimensional subvariety $\underbrace{\varphi^{-1}(0)\times\ldots\times\varphi^{-1}(0)}_{d\text{ times}}$.
As long as $d(m-1)>m(d-1)$, or equivalently $d<m=\mathrm{dim}G$,
the map $\varphi^{*d}:\mathbb{A}^{dm}\to\mathbb{A}^{m}$ cannot be
flat.

The map $\varphi^{*m}$ is flat, hence the fiber $(\varphi^{*m})^{-1}(0)$
is reduced if and only if it is generically reduced. Moreover, $(x_{1},\ldots,x_{m^{2}})$
is a smooth point of $(\varphi^{*m})^{-1}(0)$ if and only if $\varphi^{*m}$
is smooth at $(x_{1},\ldots,x_{m^{2}})$. But since $\varphi$ is
not smooth at $\varphi^{-1}(0)$, the morphism $\varphi^{*(m+k)}$
is not smooth at $\underbrace{\varphi^{-1}(0)\times\ldots\times\varphi^{-1}(0)}_{m+k\text{ times}}$,
so $(\varphi^{*(m+k)})^{-1}(0)$ is not regular in codimension $k$.
In particular, $(\varphi^{*m})^{-1}(0)$ is not reduced and $(\varphi^{*(m+1)})^{-1}(0)$
is not normal. 
\end{proof}
We conclude the section with the following observation: 
\begin{prop}
\label{prop:not smooth after convolutions}Let $\varphi:X\rightarrow G$
be a morphism from a smooth algebraic $K$-variety $X$ to an algebraic
$K$-group $G$. Assume that $\varphi$ is not smooth at $x\in X(\overline{K})$
with $\varphi(x)$ in the center of $G(\overline{K})$. Then for any
$t\in\nats$ we have that $\varphi^{*t}:X^{t}\rightarrow G$ is not
smooth at $(x,...,x)$. 
\end{prop}

\begin{proof}
Write $m:G^{t}\rightarrow G$ for the multiplication map. Since $\varphi(x)$
is central, the following holds for any $Y_{1},...,Y_{t}\in T_{x}(X)$:
\begin{align*}
d\varphi_{(x,...,x)}^{*t}(Y_{1},...,Y_{t}) & =dm_{(\varphi(x),...,\varphi(x))}\circ(d\varphi_{x},...,d\varphi_{x})(Y_{1},...,Y_{t})\\
 & =\sum_{i=1}^{t}d\varphi_{x}(Y_{i})\subseteq\mathrm{Im}(d\varphi_{x}),
\end{align*}
so $d\varphi_{(x,...,x)}^{*t}$ is not surjective. 
\end{proof}
We therefore see that by convolving a non-smooth morphism sufficiently
many times, we may achieve certain singularity properties as in Proposition
\ref{prop: properties preserved under convolution} and in Theorem
\ref{Main result}, but one can not hope in general to obtain a smooth
morphism. That said, the following example shows that such a situation
might still occur.
\begin{example}
Let $G=U_{3}(\complex)$ be the group of upper triangular unipotent
matrices with complex values and consider the morphism $\varphi:\mathbb{A}_{\complex}^{3}\rightarrow G$
given by 
\[
\varphi(x_{1},x_{2},x_{3})=\left(\begin{array}{ccc}
1 & x_{1}-1 & x_{1}x_{3}\\
0 & 1 & x_{2}\\
0 & 0 & 1
\end{array}\right).
\]
Note that $\varphi$ is not smooth at $(0,x_{2},0)$ for any $x_{2}\in\complex$,
but $\varphi^{*2}$ is already a smooth morphism: 
\[
\varphi^{*2}(x_{1},x_{2},x_{3},y_{1},y_{2},y_{3})=\left(\begin{array}{ccc}
1 & x_{1}+y_{1}-2 & \left(x_{1}x_{3}+y_{1}y_{3}+(x_{1}-1)y_{2}\right)\\
0 & 1 & x_{2}+y_{2}\\
0 & 0 & 1
\end{array}\right).
\]
\end{example}

\section{Main analytic results \label{sec:Main-analytic-result} }

In this section we prove Theorems \ref{Main model theoretic result for vector spaces},
\ref{Main model theoretic result-Varieties}, \ref{Main model theoretic result for families of vector spaces},
and \ref{Model theoretic result- families of varieties}.

In \cite[Theorem 5.2]{GH19} we have shown that given a compactly
supported motivic function $h\in\mathcal{C}(\mathbb{A}_{\rats}^{n})$,
with $\left|h_{F}\right|$ integrable for every $F\in\mathrm{Loc}_{>}$,
then $\mathcal{F}(h_{F})(y)$ decays faster than $\left|y\right|{}^{\alpha}$
for some $\alpha\in\reals_{<0}$. We deduced that any compactly supported
motivic measure $\mu=\{\mu_{F}\}_{F\in\mathrm{Loc}_{>}}$ on $\{F^{n}\}_{F\in\mathrm{Loc}_{>}}$
has a continuous density after enough self-convolutions. Using the
analytic criterion for the (FRS) property (Theorem \ref{Analytic condition for (FRS)}),
we were then able to prove Theorem \ref{Main result} in the case
where $G$ is a vector space.

If $G$ is not abelian, a problem arises as the non-commutative Fourier
transform is not as well behaved with respect to the class $\mathcal{C}(G)$
of motivic functions. Thus, a direct generalization of the proof given
in \cite[Theorem 5.2]{GH19} by bounding the decay rate of the Fourier
transform is harder. Hence, we would like to show that compactly supported,
$L^{1}$, motivic functions become continuous after sufficiently many
self convolutions, without having to estimate the decay rate of their
Fourier transform. Theorems \ref{Main model theoretic result for vector spaces}
and \ref{Main model theoretic result-Varieties} solve this problem.
Indeed, Theorem \ref{Main model theoretic result-Varieties} easily
implies the following corollary: 
\begin{cor}
\label{Cor:Main model theoretic result for algebraic groups}Let $G$
be an algebraic $\rats$-group, and let $\mu$ be a motivic measure
on $G$, such that $\mu_{F}$ is supported on $G(\mathcal{O}_{F})$
and is absolutely continuous with respect to the Haar measure $\nu_{F}$
on $G(\mathcal{O}_{F})$ with density $f_{F}$. Then there exists
$\epsilon>0$ such that $f_{F}\in L^{1+\epsilon}(G(\mathcal{O}_{F}),\nu_{F})$
for every $F\in\mathrm{Loc}_{>}$. 
\end{cor}

Notice that in the setting of Corollary \ref{Cor:Main model theoretic result for algebraic groups},
since $f_{F}\in L^{1+\epsilon}(G(\mathcal{O}_{F}),\nu_{F})$ for every 
$F\in\mathrm{Loc}_{>}$, then after enough self-convolutions it will
have a continuous density (see Lemma \ref{lem: continuous function after enough convolutions}),
as desired. Corollary \ref{Cor:Main model theoretic result for algebraic groups}
will be used in Section \ref{sec:Proof-of-the main result} to deduce
Theorem \ref{Main result}.

\subsection{\label{subsec:Proof-of-Theorem}Proof of Theorem \ref{Main model theoretic result for families of vector spaces}}

We now prove Theorem \ref{Main model theoretic result for families of vector spaces}
(which directly implies Theorem \ref{Main model theoretic result for vector spaces}): 
\begin{thm}[Theorem \ref{Main model theoretic result for families of vector spaces}]
\label{thm: Main model theoretic result for families of vector spaces v2}
Let $h\in\mathcal{C}(\mathbb{A}_{\rats}^{n}\times Y)$ be a family
of motivic functions parameterized by a $\rats$-variety $Y$, and
assume that for every $F\in\mathrm{Loc}_{>}$ we have $h_{F}|_{F^{n}\times\{y\}}\in L^{1}(F^{n})$
for every $y\in Y(F)$. Then there exists $\epsilon>0$, such that for
 every  $F\in\Loc_{>}$ we have $h_{F}|_{F^{n}\times\{y\}}\in L^{1+\epsilon}(F^{n})$,
for every $y\in Y(\rats)$.
\end{thm}

In order to prove Theorem \ref{thm: Main model theoretic result for families of vector spaces v2}
we show the following proposition: 
\begin{prop}
\label{prop: integral of h has tame form} Let $Y\subset\VF^{n'}$
be an $\Ldp$-definable set, and let $h\in\mathcal{C}(\VF^{n}\times Y)$ be a motivic function.
Then for every $F \in \Loc_{>}$ the integral $\int_{F^{n}}|h_{F}(x,y)|^{1+\epsilon}\,dx$
can be written as a finite sum of terms of the form 
\[
\sum\limits _{\vec{\xi}\in k_{F}^{\tilde{N}_{1}}}\left|B_{F,\vec{\xi},y}\right|\sum_{\lambda\in\Lambda_{y}}\sum_{(e_{1},\ldots,e_{s})\in\mathbb{Z}_{\geq0}^{s}}\left|\sum_{i=1}^{\tilde{N}_{2}}c_{i,F}(\vec{\xi},y,\lambda)q_{F}^{(\frac{d_{1}}{1+\epsilon}+d_{i1})e_{1}+\ldots+(\frac{d_{s}}{1+\epsilon}+d_{is})e_{s}}\prod\limits _{j=1}^{s}e_{j}^{b_{ij}}\right|^{1+\epsilon},
\]
where $\tilde{N}_{1},\tilde{N}_{2},s$ and $b_{ij}$ are non-negative
integers, $B\subseteq\RF^{\tilde{N}_{1}+r}\times Y$ is an $\Ldp$-definable
set, $\Lambda_{y}$ are finite Presburger sets depending on $y$,
$c_{i}$ are motivic functions and $d_{i},d_{ij}$ are integers (both  sides of the equation might be infinite for certain values of $y$ and $\epsilon$).
\end{prop}
\begin{proof}
By Definition \ref{def:motivic function}, $h_{F}(x,y)$ can be written
as 
\[
h_{F}(x,y)=\sum_{i=1}^{N}\left|Y_{i,F,x,y}\right|q_{F}^{\alpha_{i,F}(x,y)}\left(\prod_{j=1}^{N_{1}}\beta_{ij,F}(x,y)\right)\left(\prod_{j=1}^{N_{2}}\frac{1}{1-q_{F}^{a_{ij}}}\right),
\]
for $(x,y)\in F^{n}\times Y(F)\subseteq F^{n}\times F^{n'}$. By quantifier
elimination (see \cite[Theorem 4.1]{Pas89} or \cite[Theorem 2.1.1]{CL08}
for a reformulation more suitable to our needs), there exists a collection
$\{g_{i}\}_{i=1}^{N_{3}}$ of polynomials $g_{i}\in\mathbb{Z}[x_{1},\ldots,x_{n},y_{1},\ldots,y_{n'}]$,
such that each $Y_{i}\subseteq\VF^{n}\times Y\times\mathrm{RF}^{r}$
can be defined by a finite disjunction of formulas of the form 
\[
\chi(\ac(g_{1}(x,y)),\ldots,\ac(g_{N_{3}}(x,y)),\vec{\xi}')\wedge\theta(\val(g_{1}(x,y)),\ldots,\val(g_{N_{3}}(x,y))),
\]
where $\vec{\xi}'\in\mathrm{RF}^{r}$, $\chi$ is an $\mathcal{L}_{\mathrm{Res}}$
formula and $\theta$ is an $\mathcal{L}_{\mathrm{Pres}}$ formula.
Define the following level sets (parameterized by the value group
variables $\vec{m},\vec{n},\vec{k}$ and the residue field variables
$\vec{\xi}$): 
\begin{gather*}
A_{\vec{m},\vec{n}}(y):=\{x\in\VF^{n}:(\alpha_{i}(x,y),\beta_{ij}(x,y))=(m_{i},n_{ij})\text{ for all }1\leq i\leq N,~1\leq j\leq N_{1}\},\\
A'_{\vec{\xi},\vec{k}}(y):=\{x\in\VF^{n}:(\ac(g_{i}(x,y)),\val(g_{i}(x,y)))=(\xi_{i},k_{i})\text{ for all }1\leq i\leq N_{3}\}.
\end{gather*}
Clearly, for every $y\in Y(F)$ the functions $\alpha_{i,F}$ and
$\beta_{ij,F}$ are constant on each $A_{\vec{m},\vec{n},F}(y)$ and
$|Y_{i,F,x,y}|$ is constant on each $A'_{\vec{\xi},\vec{k},F}(y)$.
Furthermore, using the disjointness of the $\VG$ and $\RF$ sorts
there exists a finite partition $\ints^{N_{3}}=\bigcup_{i'=1}^{\hat{N}}D_{i'}$
such that each $|Y_{i,F,x,y}|$ is constant on $\bigcup\limits _{\vec{k}\in D_{i'}}A'_{\vec{\xi},\vec{k},F}(y)$
for each $i'$. Set $a_{F}(\vec{m},\vec{n},\vec{k},\vec{\xi},y)=\int_{A_{\vec{m},\vec{n}}(y)\cap A'_{\vec{k},\vec{\xi}}(y)}dx$.
Integrating $|h_{F}|^{1+\epsilon}$ over $F^{n}$ we get, 
\begin{equation}
\int_{F^{n}}\left|h_{F}(x,y)\right|^{1+\epsilon}dx=\sum_{\vec{\xi}\in k_{F}^{N_{3}}}\sum\limits _{i'=1}^{\hat{N}}\sum_{\begin{array}{c}
(\vec{m},\vec{n})\in\ints^{N(N_{1}+1)}\\
\vec{k}\in D_{i'}
\end{array}}a_{F}(\vec{m},\vec{n},\vec{k},\vec{\xi},y)\cdot\left|\sum_{i=1}^{N}c_{i,i',F}(\vec{\xi},y)\cdot q_{F}^{m_{i}}\cdot\prod_{j=1}^{N_{1}}n_{ij}\right|^{1+\epsilon},\label{eq:(4.1)}
\end{equation}
where $c_{i,i',F}(\vec{\xi},y)=\left|Y_{i,F,x,y}\right|\cdot\prod_{j=1}^{N_{2}}\frac{1}{1-q_{F}^{a_{ij}}}$
depends only on $\vec{\xi}$ and $y$ (as explained above).

Using \cite[Theorem 4.1]{BDOP13} repeatedly, as was done in \cite[proof of Theorem 5.1]{Pas89}
and \cite[proof of Theorem B]{BDOP13}, we can write $a_{F}(\vec{m},\vec{n},\vec{k},\vec{\xi},y)$
as a finite sum of terms of the form 
\begin{equation}
q_{F}^{-n}\sum_{\vec{\eta}\in k_{F}^{r'}}\sum_{\begin{array}{c}
\vec{l}\in\mathbb{Z}^{n}\\
\sigma(\vec{l},\vec{m},\vec{n},\vec{k},\vec{\xi},\vec{\eta},y)
\end{array}}q_{F}^{-l_{1}-\ldots-l_{n}},\label{eq:(4.2)}
\end{equation}
where $\sigma$ is an $\mathcal{L}_{\mathrm{DP}}$-formula. By quantifier
elimination, the formula $\sigma$ is equivalent to 
\[
\bigvee_{i=1}^{L}\chi_{i}(\vec{\xi},\vec{\eta},\ac(g'(y)))\wedge\theta_{i}(\vec{l},\vec{m},\vec{n},\vec{k},\val(g'(y))),
\]
where $\ac(g'(y)):=\ac(g'_{1}(y)),\ldots,\ac(g'_{N_{4}}(y))$ and
similarly for $\val(g'(y))$, and each $\chi_{i}$ is an $\mathcal{L}_{\mathrm{Res}}$-formula,
each $\theta_{i}$ is an $\mathcal{L}_{\mathrm{Pres}}$-formula and
$g'_{j}\in\ints[y_{1},\ldots,y_{n'}]$. Given $I\in\{0,1\}^{L}$,
set 
\[
B_{I}(\vec{\xi},y):=\{\vec{\eta}\in\RF^{r'}:\forall1\leq i\leq L,~\chi_{i}(\vec{\xi},\vec{\eta},y)\text{ holds }\Longleftrightarrow I(i)=1\}
\]
and $\theta_{I}:=\bigvee_{I(i)=1}\theta_{i}$. Then for each $y$
and $\vec{\xi}$, we have $\mathrm{RF}^{r'}=\underset{I\in\{0,1\}^{L}}{\bigcup}B_{I}(\vec{\xi},y)$,
and we can write (\ref{eq:(4.2)}) as 
\[
\sum_{I\in\{0,1\}^{L}}\left|B_{I,F}(\vec{\xi},y)\right|\cdot\sum_{\begin{array}{c}
\vec{l}\in\mathbb{Z}^{n}\\
\theta_{I}(\vec{l},\vec{m},\vec{n},\vec{k},y)
\end{array}}q_{F}^{-n-l_{1}-{\ldots}-l_{n}}.
\]
Set $N'=N(N_{1}+1)+N_{3}$ and $\theta_{I,i'}:=\{(\vec{l},\vec{m},\vec{n},\vec{k},y):\theta_{I}\wedge(\vec{k}\in D_{i'})\}$
then (\ref{eq:(4.1)}) can be written as a sum of finitely many expressions
of the form 
\[
\sum_{\vec{\xi}\in k_{F}^{N_{3}}}\sum_{i'=1}^{\hat{N}}\sum_{I\in\{0,1\}^{L}}\left|B_{I,F}(\vec{\xi},y)\right|\sum_{\begin{array}{c}
(\vec{l},\vec{m},\vec{n},\vec{k})\in\ints^{n+N'}\\
\theta_{I,i'}(\vec{l},\vec{m},\vec{n},\vec{k},y)
\end{array}}q_{F}^{-n-l_{1}-{\ldots}-l_{n}}\cdot\left|\sum_{i=1}^{N}c_{i,i',F}(\vec{\xi},y)\cdot q_{F}^{m_{i}}\cdot\prod_{j=1}^{N_{1}}n_{ij}\right|^{1+\epsilon}.
\]

It is enough to show that the above sum can be written as in the statement
of the proposition when we fix $I_{0}\in\{0,1\}^{L}$ and $1\leq i'_{0}\leq\hat{N}$.
Set $\theta:=\theta_{I_{0},i'_{0}}$ and $c_{i}=c_{i,i'_{0}}$, and
consider the sum 
\begin{equation}
\sum_{\begin{array}{c}
(\vec{l},\vec{m},\vec{n},\vec{k})\in\mathbb{Z}^{n+N'}\\
\theta(\vec{l},\vec{m},\vec{n},\vec{k},y)
\end{array}}q_{F}^{-l_{1}-{\ldots}-l_{n}}\left|\sum_{i=1}^{N}c_{i,F}(\vec{\xi},y)\cdot q_{F}^{m_{i}}\cdot\prod_{j=1}^{N_{1}}n_{ij}\right|^{1+\epsilon}.\label{eq:(4.3)}
\end{equation}
By uniform rectilinearization (Theorem \ref{thm:-(Uniform Rectilinearization)}),
we have the following decomposition: 
\[
\{(\vec{l},\vec{m},\vec{n},\vec{k},y)\in\ints^{n+N'}\times Y:\theta(\vec{l},\vec{m},\vec{n},\vec{k},y)\}=\bigcup\limits _{j=1}^{L'}C_{j},
\]
where for each $j$ there exists an $\mathcal{L}_{\mathrm{DP}}$-definable
isomorphism $\rho_{j}:C_{j}\to B_{j}'\subset Y\times\ints^{n+N'}$
over $Y$, such that for each $y\in Y(F)$, 
\begin{equation}
\rho_{j}|_{C_{j,y}}:C_{j,y}:=\{(\vec{l},\vec{m},\vec{n},\vec{k})\in\ints^{n+N'}:(\vec{l},\vec{m},\vec{n},\vec{k},y)\in C_{j}\}\xrightarrow{\sim}\Lambda_{y}\times\mathbb{Z}_{\geq0}^{s'_{j}}=B'_{j,y}\label{eq:(4.4)}
\end{equation}
is $\mathcal{L}_{\mathrm{Pres}}$-linear (see Definition \ref{def:-Linear function})
for some integer $s'_{j}\geq0$ depending only on $C_{j}$, and a
finite set $\Lambda_{y}\subset\ints_{\geq0}^{n+N'-s'_{j}}$ depending
on $y$. 
We can therefore write 
\begin{equation}
\rho_{j}|_{C_{j,y}}(\vec{l},\vec{m},\vec{n},\vec{k})=\beta_{j}(\vec{l},\vec{m},\vec{n},\vec{k})+\gamma_{j}(\val(g'(y))),\label{eq:(4.5)}
\end{equation}
where $\beta_{j}$ is $\mathcal{L}_{\mathrm{Pres}}$-linear and $\gamma_{j}$
is $\mathcal{L}_{\mathrm{Pres}}$-definable (functions to $\ints^{n+N'}$).
By Theorem \ref{Presburger Cell decomposition}, we may further assume
that $\gamma_{j}$ is $\mathcal{L}_{\mathrm{Pres}}$-linear. Since
the collection $\{C_{j}\}_{j=1}^{L'}$ is finite,
by breaking (\ref{eq:(4.3)}) to finitely many summands, it is enough
to consider 
\begin{equation}
\sum_{\begin{array}{c}
(\vec{l},\vec{m},\vec{n},\vec{k})\in\mathbb{Z}^{n+N'}\\
(\vec{l},\vec{m},\vec{n},\vec{k},y)\in C_{j}
\end{array}}q_{F}^{-l_{1}-\ldots-l_{n}}\left|\sum_{i=1}^{N}c_{i,F}(\vec{\xi},y)\cdot q_{F}^{m_{i}}\cdot\prod_{j=1}^{N_{1}}n_{ij}\right|^{1+\epsilon}.\label{eq:(4.55)}
\end{equation}
We thus may assume that $\rho_{j}|_{C_{j,y}}$ has the form of (\ref{eq:(4.5)}).

Fix $C:=C_{j}$, set $s':=s'_{j}$ and let $y\in Y(F)$. We can reparameterize
according to (\ref{eq:(4.4)}) and use (\ref{eq:(4.5)}) to write
the sum (\ref{eq:(4.55)}) as follows:
\begin{equation}
\sum_{\lambda\in\Lambda_{y}}\sum_{(e_{1},\ldots,e_{s'})\in\mathbb{Z}_{\geq0}^{s'}}q_{F}^{T(\vec{e},y,\lambda)}\left|\sum_{i=1}^{N}c_{i,F}(\vec{\xi},y)q_{F}^{T_{i}(\vec{e},y,\lambda)}P_{i,\lambda}(e_{1},\ldots,e_{s'},\val(g'(y)))\right|^{1+\epsilon},\label{eq:4.6}
\end{equation}
where $P_{i,\lambda}$ are polynomials with rational coefficients
and 
\[
T(\vec{e},y,\lambda)=\tau(\lambda)+\sum\limits _{j=1}^{s'}d_{j}e_{j}+\sum\limits _{j=1}^{N_{4}}d'_{j}\val(g'_{j}(y))~\text{ and }~T_{i}(\vec{e},y,\lambda)=\tau_{i}(\lambda)+\sum\limits _{j=1}^{s'}d_{ij}e_{j}+\sum\limits _{j=1}^{N_{4}}d'_{ij}\val(g'_{j}(y)),
\]
where $d_{j},d_{ij}$ are integers and $d'_{i},d'_{ij}$ are rational
numbers and $\tau(\lambda),\tau_{i}(\lambda)$ are $\ints$-valued
affine functions in $\lambda$ with rational coefficients. Now we
are done as we may write (\ref{eq:4.6}) as follows (note we abuse notation by keeping the symbols  $d_{j}$ and $d_{ij}$ although we now possibly sum over more elements):
\[
\sum_{\lambda\in\Lambda_{y}}\sum_{(e_{1},\ldots,e_{s'})\in\mathbb{Z}_{\geq0}^{s'}}\left|\sum_{i=1}^{{N_{5}}}c_{i,F}'(\vec{\xi},\lambda,y)q_{F}^{(\frac{d_{1}}{1+\epsilon}+d_{i1})e_{1}+\ldots+(\frac{d_{s'}}{1+\epsilon}+d_{is'})e_{s'}}\prod\limits _{j=1}^{s'}e_{j}^{b_{ij}}\right|^{1+\epsilon}.
\]
\end{proof}
To prove Theorem \ref{thm: Main model theoretic result for families of vector spaces v2}
we further need the following easy variant of \cite[Lemma 2.1.8]{CGH14}.

\begin{lem}
\label{lem: convergence of sums+}Let $h:\ints_{\geq0}^{s}\to\complex$
be a function of the form 
\[
h(e_{1},\ldots,e_{s})=\sum_{i=1}^{N}c_{i}q^{b_{i1}e_{1}+\ldots+b_{is}e_{s}}\prod\limits _{j=1}^{s}e_{j}^{a_{ij}}
\]
where $q\in\reals_{>1}$, $a_{ij},b_{ij},c_{i}\in\reals$, the tuples
$\{(b_{i1},\ldots,b_{is},a_{i1},\ldots,a_{is})\}_{i=1}^{N}$ are mutually
distinct and $c_{i}\neq0$. Then for every $r>0$ it holds that 
\[
\sum_{(e_{1},\ldots,e_{s})\in\mathbb{Z}_{\geq0}^{s}}\left|h(e_{1},\ldots,e_{s})\right|^{r}<\infty
\]
if and only $b_{ij}<0$ for every $i$ and $j$. 
\end{lem}

\begin{proof}[Proof of Theorem \ref{thm: Main model theoretic result for families of vector spaces v2}]
We may assume that $Y$ is affine and embedded in $\mathbb{A}_{\rats}^{n'}$.
By choosing a $\ints$-model we may assume that $Y\subseteq\VF^{n'}$
is an $\Ldp$-definable set. Using Proposition \ref{prop: integral of h has tame form}
the integral $\int_{F^{n}}\left|h_{F}(x,y)\right|^{1+\epsilon}dx$
can be written as a sum of finitely many expressions of the form 
\begin{equation}
\sum\limits _{\vec{\xi}\in k_{F}^{\tilde{N}_{1}}}\left|B_{F,\vec{\xi},y}\right|\sum_{\lambda\in\Lambda_{y}}\sum_{(e_{1},\ldots,e_{s})\in\mathbb{Z}_{\geq0}^{s}}\left|\sum_{i=1}^{{\tilde{N}_{2}}}c_{i,F}(\vec{\xi},y,\lambda)q_{F}^{(\frac{d_{1}}{1+\epsilon}+d_{i1})e_{1}+\ldots+(\frac{d_{s}}{1+\epsilon}+d_{is})e_{s}}\prod\limits _{j=1}^{s}e_{j}^{b_{ij}}\right|^{1+\epsilon}.\tag{\ensuremath{*}}\label{eq:summand}
\end{equation}
We may assume the tuples $\{(d_{i1},\ldots,d_{is},b_{i1},\ldots,b_{is})\}_{i=1}^{\tilde{N}_{2}}$
are mutually distinct and thus $\{(\frac{d_{1}}{1+\epsilon}+d_{i1},\ldots,\frac{d_{s}}{1+\epsilon}+d_{is},b_{i1},\ldots,b_{is})\}_{i=1}^{\tilde{N}_{2}}$
are mutually distinct for every $\epsilon>0$. We may furthermore
assume $\left|B_{F,\vec{\xi},y}\right|\cdot c_{i,F}(\vec{\xi},y,\lambda)\not\equiv0$
for infinitely many local fields $F\in\Loc$ with residual characteristic
as large as we like, as otherwise we may discard the $i$-th summand
from the inner summation in (\ref{eq:summand}). Since $h_{F}(x,y)$
is absolutely integrable on $F^{n}\times\{y\}$ for every $F\in\Loc_{>}$
and $y\in Y(F)$, each summand (\ref{eq:summand}) converges for every
$F\in\Loc_{>}$ and $y\in Y(F)$ when we set $\epsilon=0$. By our
assumption, for every $1\leq i\leq\tilde{N}_{2}$ there exist infinitely many $F\in\Loc_{>}$ 
with $y\in Y(F)$ such that $\left|B_{F,\vec{\xi},y}\right|\cdot c_{i,F}(\vec{\xi},y,\lambda)\neq0$
for some $\vec{\xi}$ and $\lambda$. Using Lemma \ref{lem: convergence of sums+},
it follows that $d_{j}+d_{ij}<0$ for every $1\leq j\leq s$.

Since $k_{F}^{\tilde{N}_{1}}$ and $\Lambda_{y}$ are finite for every
$F\in\Loc_{>}$ and $y\in Y(F)$ it is enough to show that for each
of the following sums there exists $\epsilon>0$ such that it converges:
\begin{equation}
\sum_{(e_{1},\ldots,e_{s})\in\mathbb{Z}_{\geq0}^{s}}\left|\sum_{i=1}^{\tilde{N}_{2}}c_{i,F}(\vec{\xi},y,\lambda)q_{F}^{(\frac{d_{1}}{1+\epsilon}+d_{i1})e_{1}+\ldots+(\frac{d_{s}}{1+\epsilon}+d_{is})e_{s}}\prod\limits _{j=1}^{s}e_{j}^{b_{ij}}\right|^{1+\epsilon}.\tag{\ensuremath{**}}\label{eq:inner summand}
\end{equation}
Using Lemma \ref{lem: convergence of sums+} again, we are done as
there exists $\epsilon>0$ satisfying the finitely many conditions
$\left\{ \frac{d_{j}}{1+\epsilon}+d_{ij}<0\right\} _{i,j}$. 
\end{proof}

\subsection{Proof of Theorems \ref{Main model theoretic result-Varieties} and
\ref{Model theoretic result- families of varieties}}

For the proof of Theorem \ref{Model theoretic result- families of varieties}
we need the following lemma: 
\begin{lem}
\label{lem:etale preserves Lp} Let $X$ and $Y$ be smooth $F$-varieties,
with $F\in\mathrm{Loc}$, $\omega_{X}$ and $\omega_{Y}$ be invertible
top forms on $X$ and $Y$ respectively, and let $\varphi:X\rightarrow Y$
be an \'etale map. Let $\mu_{F}$ be a compactly supported non-negative
measure on $X(F)$, which is absolutely continuous with respect to
$|\omega_{X}|_{F}$, with density $f$. Then $\varphi_{*}\mu_{F}$
is absolutely continuous with respect to $|\omega_{Y}|_{F}$ with
density $\widetilde{f}$, and for any $1\leq s<\infty$, we have $\widetilde{f}\in L^{s}(Y(F),|\omega_{Y}|_{F})$
if and only if $f\in L^{s}(X(F),|\omega_{X}|_{F})$. 
\end{lem}

\begin{proof}
Recall $f$ is positive. Since $\varphi$ is \'etale, it is quasi-finite
and smooth, and we have 
\[
\widetilde{f}(y)=\sum_{x\in\varphi^{-1}(y)(F)}f(x)\cdot\omega_{\varphi,F}(x),
\]
where $\omega_{\varphi,F}(x):=\left|\frac{\omega_{X}}{\varphi^{*}\omega_{Y}}\right|_{F}(x)$
is an invertible function. Let $B$ be a compact open subset of $X(F)$
such that $\mathrm{supp}(f)\subseteq B$ and $1\leq s<\infty$. Then
there exist constants $c(s),C(s)>0$ such that $c(s)<\omega_{\varphi,F}(x)^{s-1}<C(s)$
for every $x\in B$. For any $y\in Y(F)$ we may find an open compact
neighborhood $V\subseteq Y(F)$, such that $\varphi^{-1}(V)$ is a
disjoint union of $U_{1},...,U_{N}$, where each $U_{i}$ diffeomorphic
to $V$. 
Now, since $\varphi$ is quasi-finite  it holds that $\underset{y\in Y(F)}{\mathrm{sup}}\left|\varphi^{-1}(y)(F)\right|<M$
for some $M\in\nats$, and 
 since $|\cdot|^{s}$
is a quasi-norm there exists $d(s)>0$ such that 
\[
\sum_{x\in\varphi^{-1}(y)(F)}f(x)^{s}\cdot \omega_{\varphi,F}(x)^s\leq\widetilde{f}(y)^{s}\leq d(s)\sum_{x\in\varphi^{-1}(y)(F)}f(x)^{s}\cdot\omega_{\varphi,F}(x)^{s}.
\]
Since $f$ and $\widetilde{f}$ are compactly supported, the following
implies the claim: 
\begin{align*}
c(s)\int_{\varphi^{-1}(V)}f(x)^{s}\left|\omega_{X}\right|_{F} & 
<\int_{\varphi^{-1}(V)}f(x)^{s}\cdot \omega_{\varphi,F}(x)^{s-1}\left|\omega_{X}\right|_{F}
=\int_{\varphi^{-1}(V)}f(x)^{s}\cdot \omega_{\varphi,F}(x)^{s}\left|\varphi^{*}\omega_{Y}\right|_{F}\\
 & =\sum_{i=1}^{N}\int_{U_{i}}f(x)^{s}\cdot \omega_{\varphi,F}(x)^{s}\left|\varphi^{*}\omega_{Y}\right|_{F}\leq\int_{V}\widetilde{f}(y)^{s}\left|\omega_{Y}\right|_{F}\\
 & \leq d(s)\sum_{i=1}^{N}\int_{U_{i}}f(x)^{s}\cdot \omega_{\varphi,F}(x)^{s}\left|\varphi^{*}\omega_{Y}\right|_{F}\\
 & =d(s)\int_{\varphi^{-1}(V)}f(x)^{s}\cdot \omega_{\varphi,F}(x)^{s-1}\left|\omega_{X}\right|_{F}\leq C(s)d(s)\int_{\varphi^{-1}(V)}f(x)^{s}\left|\omega_{X}\right|_{F}.
\end{align*}
\end{proof}
\begin{proof}[Proof of Theorems \ref{Main model theoretic result-Varieties} and
\ref{Model theoretic result- families of varieties}]
Theorem \ref{Main model theoretic result-Varieties} follows from
Theorem \ref{Model theoretic result- families of varieties} by choosing
$Y=\mathrm{Spec}(\rats)$, so it is left to prove Theorem \ref{Model theoretic result- families of varieties}.

Since $\varphi$ is smooth we may assume that $Y$ is affine and that
$\varphi:X\rightarrow Y$ factors as follows (where $n:=\mathrm{dim}X-\mathrm{dim}Y$),
\[
\varphi:X\overset{\psi}{\rightarrow}\mathbb{A}_{\rats}^{n}\times Y\overset{\pi}{\rightarrow}Y
\]
with $\pi$ the projection to the second coordinate and $\psi$ \'etale.
For any $y\in Y(F)$ we can consider the base change $\psi|_{X_{y}}:X_{y}\rightarrow\mathbb{A}_{F}^{n}\times\{y\}$
which is an \'etale $F$-morphism. Set $f:=\psi_{*}(h)\in\mathcal{C}(\mathbb{A}_{\rats}^{n}\times Y)$
and notice that $f_{F}|_{F^{n}\times\{y\}}\in L_{\mathrm{Loc}}^{1}(F^{n})$.
We would like to find $\epsilon>0$ such that for $F\in\mathrm{Loc}_{>}$,
$f_{F}|_{F^{n}\times\{y\}}\in L_{\mathrm{Loc}}^{1+\epsilon}(F^{n})$
for any $y\in Y(F)$.

Let $\widetilde{f}\in\mathcal{C}(\mathbb{A}_{\rats}^{n}\times Y\times\mathbb{A}_{\rats}^{1})$
be the pullback of $f$ by the projection to $\mathbb{A}_{\rats}^{n}\times Y$,
and consider $\widetilde{f}\cdot1_{B}\in\mathcal{C}(\mathbb{A}_{\rats}^{n}\times Y\times\mathbb{A}_{\rats}^{1})$
with $B=\{(x,y,t):\val(x)>\val(t)\}\subseteq\VF^{n}\times Y\times\VF$,
where $\val(x)=\underset{1\leq i\leq n}{\mathrm{min}}\val(x_{i})$
for $x=(x_{1},...,x_{n})$. By Theorem \ref{Main model theoretic result for families of vector spaces},
since $(\widetilde{f_{F}}\cdot1_{B(F)})|_{F^{n}\times\{(y,t)\}}\in L^{1}(F^{n})$
for any $(y,t)\in Y(F)\times F$, then there exists $\epsilon>0$
such that for $F\in\mathrm{Loc}_{>}$ we have $(\widetilde{f_{F}}\cdot1_{B(F)})|_{F^{n}\times\{(y,t)\}}\in L^{1+\epsilon}(F^{n})$
for any $(y,t)$. But this implies that $f_{F}|_{F^{n}\times\{y\}}\in L_{\mathrm{Loc}}^{1+\epsilon}(F^{n})$
for any $y\in Y(F)$. By Lemma \ref{lem:etale preserves Lp}, $h_{F}|_{X_{y}(F)}\in L_{\mathrm{Loc}}^{1+\epsilon}$
for any $y\in Y(F)$. 
\end{proof}

\section{Proof of the main algebro-geometric results \label{sec:Proof-of-the main result}}

In this section we prove Theorems \ref{Main result}, \ref{main result for families}
and \ref{FRS on complexity}. 
\begin{thm}[{Theorem \ref{Main result}}]
\label{Main result v2} Let $X$ be a smooth $K$-variety, $G$ be
a $K$-algebraic group and let $\varphi:X\to G$ be a strongly dominant
morphism. Then there exists $N\in\mathbb{N}$ such that for any $n>N$
the $n$-th convolution power $\varphi^{*n}$ is (FRS). 
\end{thm}

We prove Theorem \ref{Main result v2} by first reducing to the case
$K=\rats$, using the same strategy as in the proof of \cite[Proposition 6.1]{GH19}.
Hence we want to show the following: 
\begin{prop}
\label{prop:reduction to Q} It is enough to prove Theorem \ref{Main result v2}
for $K=\rats$. 
\end{prop}

The proof of Proposition \ref{prop:reduction to Q} is very similar
to the proof in the case where $G$ is a vector space (\cite[Proposition 6.1]{GH19}).
The proof of \cite[Proposition 6.1]{GH19} consists of four statements
\cite[Lemmas 6.2, 6.4 and 6.5 and Proposition 6.3]{GH19}. Lemmas
6.2, 6.4 and 6.5 of \cite{GH19} hold if $G$ is any algebraic group
(i.e.~not necessarily a vector space). \cite[Proposition 6.3]{GH19}
can also be generalized to an algebraic group $G$, only this time
one needs to use a non-commutative Fourier transform. For completeness,
we prove the generalization of \cite[Proposition 6.3]{GH19}, and
thus finish the proof of Proposition \ref{prop:reduction to Q}. 
\begin{prop}[{{{Generalization of \cite[Proposition 6.3]{GH19}}}}]
\label{prop: generalization of 6.3} Let $\varphi:X\rightarrow G$
be a morphism over a finitely generated field $K'/\rats$. Assume
there exists $N\in\nats$ such that the $N$-th convolution power
$\varphi^{*N}$ is (FRS) at the diagonal point $(x,\ldots,x)$ for each $x\in X(\overline{K'})$.
Then $\varphi^{*2N}$ is (FRS). 
\end{prop}

\begin{proof}
Let $x_{1},{\ldots},x_{2N}\in X(K'')$ for some finite extension $K''/K'$
and let $K'''$ be a finite extension of $K''$. Let $p$ be a prime
such that $x_{1},{\ldots},x_{2N}\in X(\Zp)$ and $K'''\subseteq\Qp$
(there exists infinitely many such primes, see \cite[Lemma 3.15]{GH19}).
Since $\varphi^{*N}$ is (FRS) at $(x_{i},{\ldots},x_{i})$ for any
$i$, there exist Zariski open neighborhoods $(x_{i},{\ldots},x_{i})\in U_{i}\subseteq X^{N}$
such that $\varphi^{*N}$ is (FRS) at each $U_{i}$. Note that $U_{i}(\Zp)$
contains an analytic neighborhood of the form $V_{i}\times\ldots\times V_{i}$,
where $x_{i}\in V_{i}$ is open in $X(\Zp)$. By Theorem \ref{Analytic condition for (FRS)},
since $\varphi^{*N}$ is (FRS) at $V_{i}\times\ldots\times V_{i}$,
there exists a non-negative smooth measure $\mu_{i}$, with $\mathrm{supp}(\mu_{i})=V_{i}$
such that the measure 
\[
\varphi_{*}^{*N}(\mu_{i}\times{\ldots}\times\mu_{i})=\varphi_{*}(\mu_{i})*{\ldots}*\varphi_{*}(\mu_{i})
\]
has continuous density with respect to the normalized Haar measure
on $G(\Zp)$. Consider the non-commutative Fourier transform $\mathcal{F}(\varphi_{*}(\mu_{i}))$
of the measure $\varphi_{*}(\mu_{i})$ on $G(\Zp)$. Since the density
of $\varphi_{*}^{*N}(\mu_{i}\times{\ldots}\times\mu_{i})$ is continuous,
it lies in $L^{2}(G(\ints_{p}))$. By Theorem \ref{Proposition Fourier L2}(2),
we have that $\mathcal{F}(\varphi_{*}^{*2N}(\mu_{i}\times{\ldots}\times\mu_{i}))$
is in $\mathcal{H}_{1}(\hat{G(\Zp)})$ for any $i$. This implies
$\mathcal{F}(\varphi_{*}(\mu_{i}))\in\mathcal{H}_{2N}(\hat{G(\Zp)})$.
By a generalization of H\"older's inequality (Proposition \ref{prop:(Generalization-of-Holder}),
we have 
\[
\mathcal{F}(\varphi_{*}^{*N}(\mu_{1}\times{\ldots}\times\mu_{2N}))=\prod_{i=1}^{2N}\mathcal{F}(\varphi_{*}(\mu_{i}))\in\mathcal{H}_{1}(\hat{G(\Zp)}).
\]
By Theorem \ref{Fourier transform of functions} it follows that $\varphi_{*}^{*2N}(\mu_{1}\times{\ldots}\times\mu_{2N})$
has continuous density, and by Theorem \ref{Analytic condition for (FRS)}
it follows that $\varphi^{*2N}$ is (FRS) at $(x_{1},{\ldots},x_{2N})$,
as required. 
\end{proof}
We can now assume that $\varphi:X\rightarrow G$ is a strongly dominant
$\rats$-morphism. The following analytic statement, which is a straightforward
generalization of \cite[Proposition 3.16]{GH19}, is the final ingredient
needed in order to prove Theorem \ref{Main result v2}: 
\begin{prop}[{{{See \cite[Proposition 3.16]{GH19}}}}]
\label{prop:reduction to an analytic} Let $X$ be a smooth $K$-variety,
$G$ be an algebraic $K$-group, $\varphi:X\to G$ be a strongly dominant
morphism and let $\mu=\{\mu_{F}\}_{F\in\mathrm{Loc}}$ be a motivic
measure on $X$ such that for every $F\in\mathrm{Loc}$, $\mu_{F}$
is a non-negative smooth measure and $\supp(\mu_{F})=X(\mathcal{O}_{F})$
(for existence of such a measure, see \cite[Proposition 3.14]{GH19}).
Assume that there exists $n\in\mathbb{N}$ such that the measure $\varphi_{*}^{*n}(\mu_{F}\times\ldots\times\mu_{F})$
has continuous density with respect to the normalized Haar measure
on $G(F)$ for $F\in\mathrm{Loc}_{>}$. Then the map $\varphi^{*n}:X\times\ldots\times X\to G$
is (FRS). 
\end{prop}

\begin{proof}[Proof of Theorem \ref{Main result v2}]
Let $\varphi:X\rightarrow G$ and let $\mu=\{\mu_{F}\}_{F\in\mathrm{Loc}}$
be a motivic measure on $X$ as in Proposition \ref{prop:reduction to an analytic}.
By Corollary \ref{Cor:Main model theoretic result for algebraic groups}
we can find $\epsilon>0$ such the the density of $\varphi_{*}(\mu_{F})$
lies in $L^{1+\epsilon}(G(\mathcal{O}_{F}),\lambda_{F})$, where $\lambda_{F}$
is the normalized Haar measure on $G(\mathcal{O}_{F})$, for every $F\in\mathrm{Loc}_{>}$.
By Lemma \ref{lem: continuous function after enough convolutions}
this implies there exists $N(\epsilon)\in\nats$ such that $\varphi_{*}^{*N(\epsilon)}(\mu_{F}\times\ldots\times\mu_{F})$
has continuous density. By Proposition \ref{prop:reduction to an analytic}
we are done. 
\end{proof}
\begin{proof}[Proof of Theorems \ref{main result for families} and \ref{FRS on complexity}]
Theorem \ref{main result for families} is a direct generalization
of \cite[Theorem 1.9]{GH19} (see \cite[Section 7]{GH19}). Part (1)
follows from \cite[Lemma  7.4, Lemma 7.5]{GH19}. The proof of (2)
is essentially the same as the proof of \cite[Theorem 7.1 (2)]{GH19},
where we replace the vector space $V$ with $G$, and use Theorem
\ref{Main result} instead of \cite[Theorem 1.7]{GH19}. Theorem \ref{FRS on complexity}
easily follows from \cite[Corollary 7.8]{GH19}.

One can also deduce Theorem \ref{main result for families}(2) from
Theorem \ref{Model theoretic result- families of varieties} (assuming
$K=\rats$). Indeed, let $\widetilde{\varphi}:\widetilde{X}\rightarrow G\times Y$
be a $Y$-morphism as in Theorem \ref{main result for families}.
Denote by $\pi_{\widetilde{X}}:\widetilde{X}\rightarrow Y$ and $\pi:G\times Y\rightarrow Y$
the structure maps. By Theorem \ref{main result for families}(1),
generic smoothness, and by Noetherian induction, we may assume that
$\pi_{\widetilde{X}}$ is a smooth morphism, with $\widetilde{X}$
and $Y$ smooth with invertible top forms $\omega_{\widetilde{X}}$
and $\omega_{Y}$ respectively, and that $\widetilde{\varphi}_{y}:\widetilde{X}_{y}\rightarrow G$
is strongly dominant for any $y\in Y$. Let $\omega_{G}$ be an invertible
top form on $G$. Let $\mu=\{1_{\widetilde{X}(\mathcal{O}_{F})}\left|\omega_{\widetilde{X}}\right|_{F}\}_{F\in\mathrm{Loc}}$
and consider the following family of motivic measures $\mu_{y}:=\left\{1_{\widetilde{X}_{y}(\mathcal{O}_{F})}\left|\frac{\omega_{\widetilde{X}}}{\pi_{\widetilde{X}}^{*}\omega_{Y}}\right|_{F}\right\}_{F\in\mathrm{Loc}}$.
In order to prove Theorem \ref{main result for families}(2), it is
enough by Proposition \ref{prop:reduction to an analytic} to find
an $\epsilon>0$, which is independent of $y$, such that $(\widetilde{\varphi}_{y})_{*}\mu_{y,F}=h_{y,F}\left|\omega_{G}\right|_{F}$,
with $h_{y,F}\in L^{1+\epsilon}(G(\mathcal{O}_{F}),\left|\omega_{G}\right|_{F})$
for every $F\in\mathrm{Loc}_{>}$. Indeed, if we choose $N:=2N(\epsilon)$
as in Lemma \ref{lem: continuous function after enough convolutions},
then Part (2) follows by Proposition \ref{prop:(Generalization-of-Holder}.

By \cite[Corollary 3.6]{AA16}) the measure $\widetilde{\varphi}_{*}\mu_F$
is absolutely continuous with respect to $\left|\omega_{G}\wedge\omega_{Y}\right|_{F}$
and $\widetilde{\varphi}_{*}\mu_{F}=f_{F}\cdot\left|\omega_{G}\wedge\omega_{Y}\right|_{F}$,
with 
\begin{align*}
f_{F}(g,y) & =\int_{(\widetilde{X}^{\mathrm{sm},\widetilde{\varphi}}\cap\widetilde{\varphi}^{-1}(g,y))(\mathcal{O}_{F})}\left|\frac{\omega_{\widetilde{X}}}{\widetilde{\varphi}^{*}(\omega_{G}\wedge\omega_{Y})}\right|_{F}\\
 & =\int_{(\widetilde{X}_{y}^{\mathrm{sm},\widetilde{\varphi}}\cap\widetilde{\varphi}_{y}^{-1}(g))(\mathcal{O}_{F})}\left|\frac{\omega_{\widetilde{X}}}{\widetilde{\varphi}^{*}\omega_{G}\wedge\pi_{\widetilde{X}}^{*}\omega_{Y})}\right|_{F}.
\end{align*}
Since $\widetilde{\varphi}_{y}$ is strongly dominant, we further
have by \cite[Corollary 3.6]{AA16} that $h_{y,F}\in L^{1}(G(\mathcal{O}_{F}),\left|\omega_{G}\right|_{F})$.
Setting $\displaystyle \eta:=\left(\frac{\omega_{\widetilde{X}}}{\pi_{\widetilde{X}}^{*}\omega_{Y}}\right)$,
we get 
\[
h_{y,F}(g)=\int_{(\widetilde{X}_{y}^{\mathrm{sm},\widetilde{\varphi}}\cap\widetilde{\varphi}_{y}^{-1}(g))(\mathcal{O}_{F})}\left|\frac{\eta}{\widetilde{\varphi}^{*}\omega_{G}}\right|_{F}=f_{F}(g,y).
\]
By Theorem \ref{Model theoretic result- families of varieties}, we
get $h_{y,F}=f_{F}|_{G\times\{y\}}\in L^{1+\epsilon}(G(\mathcal{O}_{F}),\left|\omega_{G}\right|_{F})$
where $\epsilon>0$ does not depend on $y$. 
\end{proof}

\appendix

\section{On the decay of Fourier transform and $L^{p}$-integrability (joint
with Gady Kozma)}

\label{Appendix: L1+epsilon} In this appendix we construct a compactly
supported function $g(t)\in L^{1}(\reals)$ whose Fourier transform
decays faster than $|x|^{\alpha}$ for some $\alpha<0$ but $g\notin L^{1+\epsilon}(\reals)$
for every $\epsilon>0$. We consider $F=\reals$, but a similar construction
should work for any local field $F$. For a number $N\in\mathbb{N}$,
and a tuple $a=(a_{j})_{j=1}^{N}$ where each $a_{j}\in\{\pm1\}$
define 
\[
h_{N,a}(x):=\bigg(\sum_{j=1}^{N}a_{j}\delta_{n_{j}}\bigg)*\sinc\Big(\frac{2x}{N}\Big),
\]
where $n_{j}:=-N^{2}+(2j-1)N$, and where $\delta_{x}$ is the Dirac
delta distribution, and $\mathrm{sinc}(x)=\frac{\mathrm{sin}(\pi x)}{\pi x}$.
Recall that 
\[
\mathcal{F}\Big(\sinc\Big(\frac{2x}{N}\Big)\Big)=
\frac{N}{2}\cdot \mathrm{1}_{[-\frac{1}{N},\frac{1}{N}]}\text{ and }\mathcal{F}(\delta_{n_{j}})(t)=e^{-2\pi in_{j}t}~\forall1\leq j\leq N.
\]

\begin{prop}
\label{Prop:(A.1)} There exists $M\in\mathbb{N}$, such that for
any $\epsilon>0$ and any $M<N\in\mathbb{N}$ there exist $a_{0}^{N}=(a_{0j})_{j=1}^{N}$,
and $0<C_{1},C_{2}(\epsilon)\in\mathbb{R}$ such that the following
hold: 
\begin{enumerate}
\item $\left\Vert h_{N,a_{0}^{N}}\right\Vert _{\infty}<10\mathrm{log}(N)$
and for every $x$ with $|x|>N^{2}$, we have $|h_{N,a_{0}^{N}}(x)|\le \frac{N^{2}}{|x|-N^{2}}$. 
\vspace{0.2cm}
\item $\mathcal{F}(h_{N,a_{0}^{N}})$ is supported on $[-\frac{1}{N},\frac{1}{N}]$. 
\vspace{0.2cm}
\item $\left\Vert \mathcal{F}(h_{N,a_{0}^{N}})\right\Vert _{\infty}\le N^{2}$. 
\vspace{0.2cm}
\item $\left\Vert \mathcal{F}(h_{N,a_{0}^{N}})\right\Vert _{1}<(10+\epsilon)\sqrt{N\log N}$. 
\vspace{0.2cm}
\item $C_{1}N^{\frac{1}{2}+\frac{1}{2}\epsilon}<\left\Vert \mathcal{F}(h_{N,a_{0}^{N}})\right\Vert _{1+\epsilon}<C_{2}(\epsilon)N^{\frac{1}{2}+\epsilon}$. 
\end{enumerate}
\end{prop}

\begin{proof}
Items (1) to (3) hold for $N$ large enough and any choice of $a=(a_{j})_{j=1}^{N}$,
by direct computations. For items (4) and (5), given $t\in\{\frac{m}{N^{4}}\}_{m=-N^{3}}^{N^{3}},$
we consider the complex-valued random variable $X_{N,t}(a):=\sum\limits _{j=1}^{N}a_{j}\cdot e^{-2\pi in_{j}t},$
with respect to the uniform probability on the set of $N$-tuples
$a=(a_{j})_{j=1}^{N}$. By Bernstein's inequality, applied to the
real and imaginary parts of $X_{N,t}(a)$, we deduce that: 
\[
\mathrm{Prob}\left(\left|X_{N,t}\right|\geq s\right)\leq4\exp\left(-\frac{s^{2}}{4N}\right).
\]
\vspace{0.2cm}
Choosing $s=(4+\epsilon)\sqrt{N\log N}$ and by the union bound, we
have: 
\[
\mathrm{Prob}\left(\exists t\in\left\{ \frac{m}{N^{4}}\right\} _{m=-N^{3}}^{N^{3}}:\left|X_{N,t}(a)\right|\geq(4+\epsilon)\sqrt{N\log N}\right)\leq8N^{-(1+2\epsilon)}\tag{\ensuremath{\star}}.
\]
Notice that for every $1\leq j\leq N$ we have 
\[
\underset{t\in[-1,1]}{\mathrm{sup}}\left|\frac{d}{dt}\sum_{j=1}^{N}a_{j}e^{-2\pi in_{j}t}\right|\leq\sum_{j=1}^{N}\underset{t\in[-\frac{1}{N},\frac{1}{N}]}{\mathrm{sup}}\left|\left(-2\pi in_{j}\right)a_{j}e^{-2\pi in_{j}t}\right|\leq2\pi N^{3},
\]
and therefore for any $a=(a_{j})_{j=1}^{N}$, and any $t\in[\frac{m}{N^{4}},\frac{m+1}{N^{4}}]$:
\[
\left|\sum_{j=1}^{N}a_{j}e^{-2\pi in_{j}t}-\sum_{j=1}^{N}a_{j}e^{-2\pi in_{j}\frac{m}{N^{4}}}\right|<2\pi N^{3}\frac{(m+1)-m}{N^{4}}=\frac{2\pi}{N}\tag{\ensuremath{\star\star}}.
\]
By $(\star)$ and $(\star\star)$ we deduce (for $N$ large enough):
\begin{align*}
\mathrm{Prob} & \left(\forall t\in[-\frac{1}{N},\frac{1}{N}]:\left|X_{N,t}(a)\right|<(10+\epsilon)\sqrt{N\log N}\right)\\
 & \geq\mathrm{Prob}\left(\forall t\in\left\{ \frac{m}{N^{4}}\right\} _{m=-N^{3}}^{N^{3}}:\left|X_{N,t}(a)\right|<(4+\epsilon)\sqrt{N\log N}\right)\geq1-8N^{-(1+2\epsilon)}.
\end{align*}
We get the following (for $N$ large enough): 
\begin{align*}
\mathrm{Prob} & \left(\left\Vert \mathcal{F}(h_{N,a})\right\Vert _{1}<(10+\epsilon)\sqrt{N\log N}\right)\\
 & =\mathrm{Prob}\left(\int_{-\frac{1}{N}}^{\frac{1}{N}}\left|\sum_{j=1}^{N}a_{j}e^{-2\pi in_{j}t}\cdot\frac{N}{2}\right|dt<(10+\epsilon)\sqrt{N\log N}\right)\geq1-8N^{-(1+2\epsilon)}.
\end{align*}
Similarly, there exists a constant $C_{2}(\epsilon)$ such that for
$N$ large enough: 
\begin{equation}
\mathrm{Prob}\left(\left\{ \left\Vert \mathcal{F}(h_{N,a})\right\Vert _{1}<(10+\epsilon)\sqrt{N\log N}\right\} \cap\left\{ \left\Vert \mathcal{F}(h_{N,a})\right\Vert _{1+\epsilon}<C_{2}(\epsilon)N^{\frac{1}{2}+\epsilon}\right\} \right)\geq1-8N^{-(1+2\epsilon)}.\label{eq:(A.1)}
\end{equation}
Consider the random variable $Y(a):=\int\limits _{-\frac{1}{N}}^{\frac{1}{N}}\left|\sum\limits _{j=1}^{N}a_{j}e^{-2\pi in_{j}t}\right|^{1+\epsilon}dt$.
By the central limit theorem it can be verified that 
\[
\mathbb{E}(Y)=\int_{-\frac{1}{N}}^{\frac{1}{N}}\mathbb{E}\left(\left|\sum_{j=1}^{N}a_{j}e^{-2\pi in_{j}t}\right|^{1+\epsilon}\right)dt>\int_{-\frac{1}{N}}^{\frac{1}{N}}C_{1}N^{\frac{1}{2}(1+\epsilon)}dt=2C_{1}N^{\frac{1}{2}(\epsilon-1)},
\]
for some $1>C_{1}>0$. Notice that $0\leq Y(a)\leq2N^{\epsilon}$.
Hence, 
\begin{equation}
\mathrm{Prob}\left(Y(a)<C_{1}N^{\frac{1}{2}(\epsilon-1)}\right)\leq\mathrm{Prob}\left(Y(a)<\frac{1}{2}\mathbb{E}(Y)\right)\leq1-\frac{\mathbb{E}(Y)}{4N^{\epsilon}}<1-\frac{1}{2}C_{1}N^{-\frac{1}{2}(\epsilon+1)}.\label{eq:(A.2)}
\end{equation}
For large enough $N$ we have $\frac{1}{2}C_{1}N^{-\frac{1}{2}(1+\epsilon)}>8N^{-(1+2\epsilon)}$
and therefore by (\ref{eq:(A.1)}) and (\ref{eq:(A.2)}) we can find
$a_{0}^{N}=(a_{0j})_{j=1}^{N}$ such that $\left\Vert \mathcal{F}(h_{N,a_{0}^{N}})\right\Vert _{1}<(10+\epsilon)\sqrt{N\mathrm{log}N}$,
$\left\Vert \mathcal{F}(h_{N,a_{0}^{N}})\right\Vert _{1+\epsilon}<C_{2}(\epsilon)N^{\frac{1}{2}+\epsilon}$
and furthermore $Y(a_{0}^{N})\geq C_{1}N^{\frac{1}{2}(\epsilon-1)}$.
In particular, 
\[
\left\Vert \mathcal{F}(h_{N,a_{0}^{N}})\right\Vert _{1+\epsilon}=\left\Vert \sum_{j=1}^{N}a_{0j}e^{-2\pi in_{j}t}\frac{N}{2}\cdot1_{[-\frac{1}{N},\frac{1}{N}]}\right\Vert _{1+\epsilon}=\frac{N}{2}\cdot\left(Y(a_{0}^{N})\right)^{\frac{1}{1+\epsilon}}>C_{1}N^{1-\frac{1}{2}\frac{(1-\epsilon)}{1+\epsilon}}>C_{1}N^{\frac{1}{2}+\frac{1}{2}\epsilon}.
\]
\end{proof}
Now set $w(l):=2^{2\cdot4^{l}}$ and take $h_{w(l)}:=h_{w(l),a_{0}^{w(l)}}$
for $a_{0}^{w(l)}$ as in Proposition \ref{Prop:(A.1)} (for $l$
such that $w(l)<M$, take $h_{w(l)}=0$). Set 
\[
H_{l}:=\frac{h_{w(l)}(x)}{\sqrt{w(l)}\log(w(l))\cdot l^{2}}
\]
and consider $f(x):=\sum\limits _{l=1}^{\infty}H_{l}$. The first
property of $h_{n,a_{0}^{N}}$ gives 
\[
|H_{l}(x)|\le(\max\{1,|x|\})^{-1/8}w(l)^{-1/8}.
\]
Hence the sum converges (both pointwise and as Schwartz distributions),
$f$ is well defined and satisfies $|f(x)|\le C|x|^{-1/8}$ and $g:=\mathcal{F}(f)=\sum\mathcal{F}(H_{l})$.
By the second property of $h_{n,a_{0}^{N}}$ $g$ is compactly supported.
We will now show that it is $L^{1}$ but not $L^{1+\epsilon}$: 
\[
\left\Vert g\right\Vert _{1}\leq\sum_{l=1}^{\infty}\left\Vert \mathcal{F}(H_{l})\right\Vert _{1}<\sum_{l=1}^{\infty}\frac{1}{l^{2}}\cdot\frac{(10+\epsilon)\sqrt{w(l)}\log(w(l))}{\sqrt{w(l)}\log(w(l))}=\sum_{l=1}^{\infty}\frac{(10+\epsilon)}{l^{2}}<\infty.
\]
By Proposition \ref{Prop:(A.1)}, there exists $D_{1},D_{2}\in\reals_{>0}$
such that for any $l$ large enough 
\begin{equation}
D_{2}\cdot2^{2\cdot4^{l}\cdot\epsilon}\geq\left\Vert \mathcal{F}(H_{l})\right\Vert _{1+\epsilon}\geq\frac{C_{1}w(l)^{\frac{1}{2}+\frac{1}{2}\epsilon}}{l^{2}\sqrt{w(l)}\log(w(l))}\geq D_{1}\cdot2^{3\cdot4^{l-1}\cdot\epsilon}>4\left\Vert \sum_{j=1}^{l-1}\mathcal{F}(H_{j})\right\Vert _{1+\epsilon}.\label{eq:(A.3)}
\end{equation}
Recall that for each $l$, $\mathcal{F}(H_{l})$ is supported on $[-\frac{1}{w(l)},\frac{1}{w(l)}]$.
Denote $S_{l}:=[-\frac{1}{w(l)},-\frac{1}{w(l+1)})\cup(\frac{1}{w(l+1)},\frac{1}{w(l)}]$
and notice that since $\left|\mathcal{F}(H_{l})\right|^{1+\epsilon}$
is bounded by $w(l)^{2(1+\epsilon)}\ll w(l+1)$ it follows that 
\begin{equation}
\left\Vert \mathrm{1}_{S_{l}}\cdot\mathcal{F}(H_{l})\right\Vert _{1+\epsilon}^{1+\epsilon}>\frac{1}{2}\left\Vert \mathcal{F}(H_{l})\right\Vert _{1+\epsilon}^{1+\epsilon},\label{eq:(A.4)}
\end{equation}
for $l$ large enough. Notice that $\mathcal{F}(H_{j})|_{S_{l}}=0$
for any $j>l$. Choose $l_{0}$ such that (\ref{eq:(A.3)}) and (\ref{eq:(A.4)})
for any $l\geq l_{0}$. To show that $g$ is not $L^{1+\epsilon}$-integrable,
it is enough to show that $\sum\limits _{l=l_{0}}^{\infty}\mathcal{F}(H_{l})$
is not $L^{1+\epsilon}$-integrable. Finally, by the fact that $\left|~ \cdot ~\right|^{1+\epsilon}$
is a quasi-norm, by (\ref{eq:(A.3)}) and by (\ref{eq:(A.4)}), we
have: 
\begin{align*}
\left\Vert \sum\limits _{l=l_{0}}^{\infty}\mathcal{F}(H_{l})\right\Vert _{1+\epsilon}^{1+\epsilon} & \ge\sum_{j=l_{0}}^{\infty}\int_{S_{j}}\left|\sum\limits _{l=l_{0}}^{\infty}\mathcal{F}(H_{l})\right|^{1+\epsilon}=\sum_{j=l_{0}}^{\infty}\int_{S_{j}}\left|\sum\limits _{l=l_{0}}^{j}\mathcal{F}(H_{l})\right|^{1+\epsilon}\\
 & \geq\sum_{j=l_{0}}^{\infty}\left(2^{-\epsilon}\int_{S_{j}}\left|\mathcal{F}(H_{j})\right|^{1+\epsilon}-\left\Vert \sum_{l=l_{0}}^{j-1}\mathcal{F}(H_{l})\right\Vert _{1+\epsilon}^{1+\epsilon}\right)\\
 & >\sum_{j=l_{0}}^{\infty}\left(2^{-(1+\epsilon)}\left\Vert \mathcal{F}(H_{j})\right\Vert _{1+\epsilon}^{1+\epsilon}-\left\Vert \sum_{l=l_{0}}^{j-1}\mathcal{F}(H_{l})\right\Vert _{1+\epsilon}^{1+\epsilon}\right)\geq\sum_{j=l_{0}}^{\infty}2^{-(2+\epsilon)}\left\Vert \mathcal{F}(H_{j})\right\Vert _{1+\epsilon}^{1+\epsilon}=\infty,
\end{align*}
and thus $g$ is not $L^{1+\epsilon}$-integrable.

\bibliographystyle{alpha}
\bibliography{bibfile}

\end{document}